\documentclass[oneside,12pt]{amsart}

\usepackage{amscd,amsmath,latexsym,bm,pdfpages,verbatim,amsthm}
\usepackage[mathcal]{euscript}
\usepackage{graphicx}
\usepackage{amssymb}          
\usepackage{tikz}
\usepackage{tikz-cd}
\usetikzlibrary{cd}
\usepackage{mathtools,enumerate}
\usepackage[linktocpage=true]{hyperref}

\newtheorem{thm}{Theorem}[section]
\newtheorem{cor}[thm]{Corollary}
\newtheorem{lem}[thm]{Lemma}

\theoremstyle{definition}
\newtheorem{defin}[thm]{Definition}
\theoremstyle{definition}

\theoremstyle{definition}
\newtheorem{exm}[thm]{Example}

\newtheorem{remark}[thm]{Remark}

\theoremstyle{remark}
\newtheorem*{rem}{Remark}

\DeclareMathOperator*{\motimes}{\text{\raisebox{0.25ex}{\scalebox{0.7}{$\bigotimes$}}}}
\DeclareMathOperator*{\moplus}{\text{\raisebox{0.28ex}{\scalebox{0.75}{$\bigoplus$}}}}
\DeclareMathOperator*{\mvee}{\text{\raisebox{0.28ex}{\scalebox{0.7}{$\bigvee$}}}}
\DeclareMathOperator*{\msetm}{\text{\raisebox{0.28ex}{\scalebox{0.7}{$\setminus$}}}}
\DeclareMathOperator*{\mwedge}{\text{\raisebox{0.28ex}{\scalebox{0.7}{$\bigwedge$}}}}
\DeclareMathOperator*{\mprod}{\text{\raisebox{0.28ex}{\scalebox{0.67}{$\prod$}}}}
\DeclareMathOperator*{\lprod}{\text{\raisebox{0.28ex}{\scalebox{0.76}{$\prod$}}}}
\DeclareMathOperator*{\mcup}{\text{\raisebox{0.28ex}{\scalebox{0.78}{$\bigcup$}}}}
\newcommand{\mz}{\mathbb{S}^{0}}
\newcommand{\umz}{\underline{\mathbb{S}}^{0}}

\begin{document}
\def\X#1#2{r(v^{#2}\ds{\prod_{i \in #1}}{x_{i}})}
\def\skp#1{\vskip#1cm\relax}
\def\mr#1{\mathring{#1}}
\def\nm#1{\mbox{{\normalsize $#1$}}}
\def\lm#1{\mbox{{\large $#1$}}}
\def\block{\rule{2.4mm}{2.4mm}}
\def\nd{\noindent}
\def\becomes{\colon\hspace{-2,5mm}=}
\def\ds{\displaystyle}
\def\red{\color{red}}
\def\blue{\color{blue}}
\def\black{\color{black}}
\def\s{\sigma}
\numberwithin{equation}{section}
\title[Cohomology of polyhedral products]{A Cartan formula for the cohomology of  polyhedral products and its
application to the ring structure}

\author[A.~Bahri]{A.~Bahri}
\address{Department of Mathematics,
Rider University, Lawrenceville, NJ 08648, U.S.A.}
\email{bahri@rider.edu}

\author[M.~Bendersky]{M.~Bendersky}
\address{Department of Mathematics
CUNY,  East 695 Park Avenue New York, NY 10065, U.S.A.}
\email{mbenders@hunter.cuny.edu}

\author[F.~R.~Cohen]{F.~R.~Cohen}
\address{Department of Mathematics,
University of Rochester, Rochester, NY 14625, U.S.A.}
\email{cohf@math.rochester.edu}

\author[S.~Gitler]{S.~Gitler}
\address{}


\begin{abstract}
 We give a geometric method for determining the cohomology groups and the product structure of
a polyhedral product $Z\big(K;(\underline{X}, \underline{A})\big)$, under suitable 
freeness conditions  or with coefficients taken in a  field $k$. This is done by considering first the special case 
$(X_{i},A_{i }) = (B_{i}\vee C_{i}, B_{i}\vee E_{i})$ for all $i$, where $E_{i}\hookrightarrow C_{i} $ is a 
null homotopic inclusion, and then deriving a decomposition for these polyhedral products which resembles 
a Cartan  formula. The result is then generalized to arbitrary  $Z\big(K;(\underline{X}, \underline{A})\big)$.
This leads to a direct computation of the Hilbert-Poincar\'e series for $Z\big(K;(\underline{X}, \underline{A})\big)$.
Other applications are included 

The product structure on $\widetilde{H}^{\ast}\big(Z\big(K;(\underline{X}, \underline{A})\big)\big)$ is described in terms
of the additive generators, labelled via the Cartan decomposition. This is done by combining with the Wedge Lemma,
the description in \cite{bbcg3} of the $\ast$-product structure of a polyhedral product which is induced from the stable 
splitting
\skp{0.1}
The description given suffices to enable explicit calculations.

\end{abstract}

\maketitle
\tableofcontents
\section{\hspace{2mm}Introduction}\label{sec:introduction}
Polyhedral products $Z\big(K;(\underline{X}, \underline{A})\big)$, \cite{bbcg1},  are defined
for a simplicial complex $K$ on the vertex set $[m] = \{1,2,\ldots,m\}$, and a family of pointed CW pairs
$$(\underline{X}, \underline{A}) = \big\{(X_{i},A_{i}): i =1,2,\ldots,m\big\}.$$

{\let\thefootnote\relax\footnote{\hspace{-4.5mm}This work was supported in part by grant 426160 from Simons Foundation. 
The authors are grateful to the Fields Institute for a conducive environment during the 
Thematic Program: Toric Topology and Polyhedral Products. \mbox{The first author acknowledges Rider 
University for a Spring 2020 research leave.}}}

\nd They are natural subspaces of the Cartesian product $X_{1}\times X_{2}\times\cdots\times X_{m}$, in
such a way that if $K = \Delta^{m-1}$, the $(m-1)$-simplex, then 
$$Z\big(K;(\underline{X}, \underline{A})\big) = X_{1}\times X_{2}\times\cdots\times X_{m}.$$
More specifically, we consider $K$ to be a category
where the objects are the simplices of $K$ and the morphisms $d_{\sigma,\tau}$ are the inclusions 
$\sigma \subset \tau$.  A  polyhedral product is given as the colimit of a diagram 
$D_{(\underline{X}, \underline{A})}: K \to CW_{\ast}$, where at each $\sigma \in K$, we set 
\begin{equation}\label{eqn:d.sigma}
D_{(\underline{X}, \underline{A})}(\sigma) =\lprod^m_{i=1}W_i,\quad {\rm where}\quad
W_i=\left\{\begin{array}{lcl}
X_i  &{\rm if} & i\in \sigma\\
A_i &{\rm if} & i\in [m]-\sigma.
\end{array}\right.
\end{equation}

\nd  Here, the colimit is a union given by 
$$Z(K; (\underline{X}, \underline{A})) = \mcup_{\sigma \in K}D_{(\underline{X}, \underline{A})}(\sigma),$$
 but the full colimit structure is used heavily in  the development of the elementary theory.
Notice that when $\sigma \subset \tau$  then 
$D_{(\underline{X}, \underline{A})}(\sigma) \subseteq D_{(\underline{X}, \underline{A})}(\tau)$. 
In the case that $K$ itself is a simplex, 
$$Z(K; (\underline{X}, \underline{A}))\; =\; \lprod\limits_{i=1}^{m}{X_i}.$$ 

Polyhedral products were formulated first for the case $(X_i,A_i) = (D^2,S^1)$  by V.~Buchstaber
and T.~Panov in \cite{bp3}; they called their spaces {\em moment-angle complexes\/}.

\skp{0.2}
In a way entirely similar to that above, a related space $\widehat{Z}(K; (\underline{X}, \underline{A}))$, called the 
{\em polyhedral  smash product\/},  is defined by replacing the Cartesian product everywhere above by the 
smash product. That is,
$$\widehat{D}_{(\underline{X}, \underline{A})}(\sigma) =\mwedge ^m_{i=1}W_i \quad {\rm and} \quad
\widehat{Z}(K; (\underline{X}, \underline{A})) = 
\mcup_{\sigma \in K}\widehat{D}_{(\underline{X}, \underline{A})}(\sigma)$$
\nd  with
$$\widehat{Z}(K; (\underline{X}, \underline{A})) \; \subseteq \; \mwedge\limits_{i=1}^{m}{X_i}.$$
The polyhedral smash product is related to the polyhedral product by the stable decomposition 
discussed  in \cite{bbcg1} and \cite{bbcg2}. We denote by $(\underline{X}, \underline{A})_J$
the restricted family of CW-pairs $\big\{(X_j,A_j\big)\}_{j\in J}$, and by $K_J$, the full subcomplex
on $J \subset [m]$. 
\begin{thm}\cite[Theorem 2.10]{bbcg2}\label{thm:bbcgsplitting}
Let $K$ be an abstract simplicial complex  on vertices $[m]$.  Given a family $\{(X_j , A_j)\}_{j=1}^m$ of
pointed  pairs  of CW-complexes, there is a natural pointed homotopy equivalence 
\begin{equation}\label{eqn:splitting}
H\colon \Sigma{\big(Z\big(K;(\underline{X},\underline{A})\big)\big)} \longrightarrow 
\Sigma\big(\mvee_{J\subseteq [m]}\widehat{Z}\big(K_J; (\underline{X}, \underline{A})_J\big)\big)
\end{equation}
\end{thm}

In many of the most important cases, the spaces $\widehat{Z}\big(K_J; (\underline{X}, \underline{A})_J\big)\big)$
can be identified explicitly, \cite{bbcg2}. Moreover, it is shown by the authors in \cite{bbcg3} that for based
CW pairs $(\underline{X},\underline{A})$, the product 
structure on the cohomology of the polyhedral product has
a canonical formulation in terms of partial diagonal maps on these spaces, (reviewed in Section \ref{subsec:back}).

Aside from the various unstable and stable splitting theorems, \cite{bbcg1,gt,gtshifted,ikshifted,ik5},
there is an extensive history of  computations of the cohomology 
groups and rings of various families of polyhedral products,  
\cite[Sections 5, 8 and 11]{bbcgsurvey}, see also \cite{ldm1,franz,bbp,zheng,zheng2,bbcg10,cai1,caichoi}. 

Some very early calculations of the cohomology of certain moment-angle complexes, (the case 
$(X_i,A_i) = (D^2,S^1)$ for all $i = 1,2,\ldots,m$), appeared in the work of Santiago L\'opez de Medrano 
\cite{ldm1}, though at that time the spaces he studied were
not recognized to have the structure of a moment-angle complex. 
The cohomology algebras of all moment-angle complexes was computed first by M.~Franz \cite{franz} and by
 I.~Baskakov, V.~Buchstaber and T.~Panov in \cite{bbp}.

The cohomology of the polyhedral product  $Z\big(K;  (\underline{X}, \underline{A}) \big)$,  for 
$(\underline{X},\underline{A})$, satisfying certain freeness conditions, (coefficients in a field $k$ for example),  
was computed using 
a spectral sequence  by the authors in \cite{bbcg10}. A computation using different methods by Q.~Zheng can 
be found in \cite{zheng, zheng2}. 

The description herein of the product structure of 
$\widetilde{H}^{*}\big(\widehat{Z}(K;(\underline{X},\underline{A}))\big)$ is the most explicit of which 
we know. 

As announced in \cite[Section 12]{bbcgsurvey},  one goal of the current paper is to show that for certain pairs 
$ (\underline{U}, \underline{V})$, called {\em wedge decomposable\/}, the algebraic decomposition given by 
the spectral sequence calculation \cite[Theorem $5.4$]{bbcg10}  is a consequence of an underlying geometric splitting. 
{\em Moreover, the results of this observation extend to general based CW-pairs of finite type.}

This paper is partly a revised version of the authors' unpublished preprint from 2014, which in turn originated from an 
earlier preprint from 2010.   In addition, the results have been extended now to describe the product structure of the 
cohomology. 

We begin in Section \ref{sec:wdpairs} by defining {\em wedge decomposable pairs\/} $(\underline{U},\underline{V})$
and deriving for them
an explicit decomposition of the polyhedral product into a wedge of much simpler spaces, (Theorem \ref{thm:main}
and Corollary \ref{cor:wedge}). In particular, this allows us to identify explicit additive generators for 
$H^{\ast}\big( Z\big(K;  (\underline{U}, \underline{V})\big)$.  The proof in Section
\ref{sec:thmmain} is an induction based on a filtration of the polyhedral product which is introduced in Section 
\ref{sec:filtration}. 

A notion of {\em cohomological wedge decomposability\/} is introduced in Section \ref{sec:cohomwdl}
and is used in Sections \ref{sec:wdcofibl} to \ref{sec:general} to extend the results described above for 
$H^{\ast}\big( Z\big(K;  (\underline{U}, \underline{V})\big)$, 
to the cohomology of arbitrary polyhedral products $Z\big(K;(\underline{X},\underline{A})\big)$. Given a collection
$ (\underline{X},\underline{A})$, we associate to it wedge decomposable pairs $(\underline{U}, \underline{V})$ 
such that there is an isomorphism of groups
$$\theta_{(\underline{U},\underline{V})} \colon 
\widetilde{H}^{*}\big(\widehat{Z}(K;(\underline{U},\underline{V}))\big) \longrightarrow
\widetilde{H}^{*}\big(\widehat{Z}(K;(\underline{X},\underline{A}))\big)$$

{\em Effectively, this result (Theorem  \ref{thm:cartan2})  gives an explicit method for labelling additive generators of 
the group
$H^{\ast}\big( Z\big(K;  (\underline{X}, \underline{A}));k\big)$, in terms of an appropriate choice of generators for 
$H^{\ast}(X)$ and $H^{\ast}(A)$, and the link structure of the simplicial complex $K$\/}. 
\skp{0.1}
\nd Applications of the additive results  comprise Sections \ref{sec:hpseries} and \ref{sec:applications}.

 A discussion of the ring structure  of $H^{\ast}\big(Z(K;(\underline{X},\underline{A}))\big)$ occupies 
Section \ref{sec:products}. The main result is Theorem \ref{thm:mainprodthm} which,  for two classes
$\theta_{(\underline{U},\underline{V})}(u)$ and $\theta_{(\underline{U},\underline{V})}(v)$ 
in  $\widetilde{H}^{*}\big(\widehat{Z}(K;(\underline{X},\underline{A}))\big)$, allows for
the explicit determination of the product 
$$\theta_{(\underline{U},\underline{V})}(u)\hspace{0.5mm}\cdot\hspace{0.5mm}\theta_{(\underline{U},\underline{V})}(v) 
\;\in \;\widetilde{H}^{*}\big(\widehat{Z}(K;(\underline{X},\underline{A}))\big).$$ 
\nd (Of course, in general, this will not be the class $\theta_{(\underline{U},\underline{V})}(u \cdot v)$.) 

The discussion of ring structure begins in subsection \ref{subsec:back} with a review from \cite{bbcg3} of the 
relationship between the stable splitting of the polyhedral product and the cohomology cup product.
The Cartan decomposition of Theorem \ref{thm:main}  enables access to the wedge lemma which decomposes the 
polyhedral smash product for a wedge decomposable pair, into spaces which are indexed by links. We observe  in 
subsections \ref{subsec:linkprods} and \ref{sec:wdpairsprod} the way in which this leads to a concise description
of cohomology product in the case of wedge decomposable pairs, Theorem \ref{thm:product for wdpair} and
Corollary \ref{cor:wdecompprod}.

Included also in subsection \ref{subsec:linkprods}, is a new method for computing the $\ast$-product of the links
which  index summands in the wedge lemma decomposition of the polyhedral smash product. This uses the 
polyhedral product $Z\big(K;(\mz,S^0)\big)$ where $\mz \simeq S^0$  is defined by \eqref{eqn:s0}. 

Finally, in subsection \ref{subsec:prodxa}, we describe the way in which the results for wedge decomposable pairs
can be extended to arbitrary pairs $(\underline{X}, \underline{A})$. The fact that all pairs are cohomologically
wedge decomposable, suffices to do product calculations, the main preoccupation being the tracking of changes
to the indexing links generated by a mixing of terms arising from  cup products in $\widetilde{H}^{\ast}(X_i)$
and  $\widetilde{H}^{\ast}(A_i)$.

\section{\hspace{2mm}The polyhedral product of wedge decomposable pairs}\label{sec:wdpairs}
We begin with a definition.

\begin{defin}\label{def:wedgedecomp}
The special family of CW pairs $(\underline{U}, \underline{V}) = (\underline{B\vee C}, \underline{B\vee E})$
satisfying   $(U_{i},V_{i }) = (B_{i}\vee C_{i}, B_{i}\vee E_{i})$ for all $i$, where $E_{i}\hookrightarrow C_{i} $ is a
null homotopic inclusion, is called {\em wedge decomposable\/}.
\end{defin} 
The fact that the smash product distributes over wedges of spaces, leads to the characterization of the smash polyhedral 
product in a way which resembles a {\em Cartan formula\/}.

\begin{thm}\label{thm:main} (Cartan Formula)
Let $(\underline{U}, \underline{V}) = (\underline{B\vee C}, \underline{B\vee E})$ be a wedge decomposable
pair, then there is a homotopy equivalence
$$\widehat{Z}\big(K;(\underline{U}, \underline{V})\big) \longrightarrow 
\mvee_{I\leq [m]}\Big(\widehat{Z}\big(K_{I};(\underline{C},\underline{E})_I\big) \wedge
\widehat{Z}\big(K_{[m]- I};(\underline{B},\underline{B})_{[m]-I}\big)\Big)$$

\nd which is natural with respect to maps of decomposable pairs.  Of course, 
$$\widehat{Z}\big(K_{[m]- I};(\underline{B},\underline{B})_{[m]-I}\big) \;=\;  \mwedge_{j\; \in\; [m]-I}B_j$$
with the convention that 
$$\widehat{Z}\big(K_{\varnothing};(\underline{B},\underline{B})_{\varnothing}\big),\;
\widehat{Z}\big(K_{\varnothing};(\underline{C},\underline{E})_{\varnothing}\big)
 \;\; {\rm{and}}\;\; \widehat{Z}\big(K_{I};(\varnothing,\varnothing)_{I}\big)
 \;=\; S^0.$$
\end{thm}
\skp{0.3}
We can decompose $\widehat{Z}\big(K;(\underline{U}, \underline{V})\big)$ further by applying
(a generalization of) the Wedge Lemma. We recall first the definition of a link.
\begin{defin}\label{defn:link}
For $\sigma$  a simplex in a simplicial complex $\mathcal{K}$, $\text{lk}_{\sigma}(\mathcal{K})$  
{\em the link of\/} $\sigma$ {\em in\/} $\mathcal{K}$, is defined to be the simplicial complex  for which
$$ \tau \in \text{lk}_{\sigma}(K)\quad \text{if and only if}\quad \tau \cup \sigma \in \mathcal{K}.$$
\end{defin} 
\begin{thm}\label{thm:wlemma} \cite[Theorem $2.12$]{bbcg1}, \cite[Lemma 1.8]{zz}
Let $\mathcal{K}$ be a simplicial complex on $[m]$ and $(\underline{C}, \underline{E})$ a family of CW pairs
satisfying $E_i \hookrightarrow C_i$ is null homotopic for all $i$ then
$$\widehat{Z}(\mathcal{K}; (\underline{C}, \underline{E})) \simeq \mvee_{\sigma \in \mathcal{K}}
|\Delta(\overline{\mathcal{K}})_{<\sigma}|\ast
\widehat{D}_{\underline{C},\underline{E}}^{[m]}(\sigma)$$
\nd where $|\Delta(\overline{\mathcal{K}})_{<\sigma}| \cong |\text{lk}_{\sigma}(\mathcal{K})|$, the 
realization of the link of $\sigma$ in the
simplicial complex $\mathcal{K}$  and  
\begin{equation}\label{eqn:ced.sigma}
\widehat{D}_{\underline{C},\underline{E}}^{[m]}(\sigma)  =\mwedge^{m}_{j=1}W_{i_{j}},\quad {\rm with}\quad
W_{i_{j}}=\left\{\begin{array}{lcl}
C_{i_{j}}  &{\rm if} & i_{j}\in \sigma\\
E_{i_{j}}  &{\rm if} & i_{j}\in [m]-\sigma.
\end{array}\right.
\end{equation} \hfill $\square$
\end{thm}
Applying this to the decomposition of Theorem \ref{thm:main}, we get

\begin{cor}\label{cor:wedge} There is a homotopy equivalence
$$\widehat{Z}\big(K;(\underline{U}, \underline{V})\big) \longrightarrow 
\mvee_{I\leq [m]}\Big(\big(\mvee_{\sigma \in K_{I}} |lk_{\sigma}(K_{I})|\ast
\widehat{D}_{\underline{C},\underline{E}}^{I}(\sigma)\big) \wedge
\widehat{Z}\big(K_{[m]- I};(\underline{B},\underline{B})_{[m]-I}\big)\Big).$$

\nd where $\widehat{D}_{\underline{C},\underline{E}}^{I}(\sigma)$ is as in \eqref{eqn:ced.sigma} with
$I$ replacing $[m]$.
\end{cor}

\nd Combined with Theorem \ref{thm:bbcgsplitting}, this gives a complete
description of the topological spaces $Z\big(K;(\underline{U}, \underline{V})\big)$ for wedge decomposable pairs
$(\underline{U}, \underline{V})\big)$. 
\skp{0.1}
The case $E_i \simeq \ast$ simplifies further by \cite[Theorem $2.15$]{bbcg2} to give the next corollary.
\begin{cor}\label{cor:E.a.point}
For wedge decomposable pairs of the form $(\underline{B\vee C}, \underline{B})$, corresponding to $E_i \simeq \ast$
for all $i = 1,2,\ldots,m$, there are homotopy equivalences
$$\widehat{Z}\big(K_{I};(\underline{C},\underline{E})_I\big) \simeq \widehat{Z}\big(K_{I};(\underline{C},\ast)_I\big)
\simeq  \widehat{C}^I,$$
and so Theorem \ref{thm:main} gives $\widehat{Z}\big(K;(\underline{B\vee C}, \underline{B})\big) \simeq
\mvee_{I\leq [m]}\big(\widehat{C}^I\wedge \widehat{B}^{([m]-I)}\big)$. \hfill $\square$
\end{cor}

\nd Notice here that the Poincar\'e series for the space $\widehat{Z}\big(K;(\underline{B\vee C}, \underline{B})\big)$
follows easily from Corollary \ref{cor:E.a.point}. 

\begin{rem} In comparing these observations with  \cite[Theorem $5.4$]{bbcg10}, notice that the links appear in the terms
$\widehat{Z}\big(K_{I};(\underline{C},\underline{E})_I\big)$. Also, while  Theorem \ref{thm:main} and Corollary \ref{cor:wedge} 
give a geometric underpinning for the cohomology calculation in \cite[Theorem $5.4$]{bbcg10} for wedge decomposable pairs, 
the geometric splitting does not require that $E,B$ or $C$ have torsion-free cohomology
\end{rem}

\section{\hspace{2mm}A filtration}\label{sec:filtration}
We begin by reviewing the filtration on polyhedral products used for the spectral sequence calculation in \cite{bbcg10}.
Following \cite[Section 2]{bbcg10}, where more details can be found.  The  length-lexicographical 
ordering, ({\em shortlex}), on the faces  of the $(m-1)$-simplex $\Delta[m-1]$ is induced by an ordering on the vertices.
This is the left lexicographical ordering on strings of varying lengths with shorter strings taking
precedence.
The ordering gives a filtration on $\Delta[m-1]$ by $$F_{t}(\Delta[m-1]) = \mcup_{s\leq t}\sigma_s.$$

\nd In turn, this gives a total ordering on the simplices of a simplicial $K$ on $m$ vertices
\begin{equation}\label{eqn:simplexordering}
\sigma_{0} = \varnothing< \sigma_{1}<  \sigma_{2}<\ldots <\sigma_{t}<\ldots< \sigma_{s}
\end{equation}
\nd via the natural inclusion $$K\; \subset \; \Delta[m-1].$$
This is filtration preserving in the sense that $F_{t}K = K \cap F_{t}\Delta[m-1]$. 
\begin{exm}
Consider $[m] = [3]$ and 
$$K  = \big\{\phi, \{v_1\},\{v_2\},\{v_3\}, \big\{\{v_1\},\{v_3\}\big\}, \big\{\{v_2\},\{v_3\}\big\}\big\}$$ 
with the realization consisting of two edges with a common vertex.
Here the  length-lexicographical ordering on the two-simplex $\Delta[2]$ is 
$$\phi <  v_1<v_2<v_3<v_1v_2<v_1v_3<v_2v_3<v_1v_2v_3$$
and so the induced ordering on $K$ is
$$\phi <  v_1<v_2<v_3<v_1v_3<v_2v_3\/.$$
\end{exm}
\begin{rem}
 Notice that if $t<m$, then $F_{t}K$ will contain  {\em ghost\/} vertices, that is, vertices which are in $[m]$ but are not 
considered simplices, They do however label Cartesian product factors in the polyhedral product. 
\end{rem}
As described in \cite[Section 2]{bbcg10},
this induces a natural filtration on the polyhedral product $Z\big(K;(\underline{X}, \underline{A})\big)$ and the smash 
polyhedral product  $\widehat{Z}\big(K;(\underline{X}, \underline{A})\big)$ as follows:
$$F_{t}{Z}(K;(\underline{X},\underline{A})) \;=\; \mcup_{k \leq t}D_{(\underline{X}, \underline{A})}({\sigma_k})
\quad {\rm{and}} \quad F_{t}\widehat{Z}(K;(\underline{X},\underline{A})) = \mcup_{k \leq t}
\widehat{D}_{\underline{X},\underline{A}}(\sigma_{k}).$$

\nd Notice also that the filtration satisfies
\begin{equation}\label{eqn:flitration}
F_{t}\widehat{Z}(K;(\underline{X},\underline{A})) \;=\; \widehat{Z}(F_{t}K;(\underline{X},\underline{A})).
\end{equation}
\skp{0.3}
\section{\hspace{2mm}The proof of Theorem \ref{thm:main}}\label{sec:thmmain}
Let the family of CW pairs $(\underline{U}, \underline{V})$ be wedge decomposable as in 
Definition \ref{def:wedgedecomp}.
 We begin by checking that Theorem \ref{thm:main} holds for $F_{0}\widehat{Z}(K;(\underline{U},\underline{V}))$.
In this case $F_{0}K$ consists of the empty simplex, (the boundary of a point), and $m-1$ ghost vertices.
So, 
\begin{align}\label{eqn:F0}
\widehat{Z}(F_{0}K;(\underline{U},\underline{V})) &\;=\; V_1 \wedge V_2\wedge  \cdots \wedge V_m\\
&\;=\; (B_1 \vee E_1) \wedge (B_2 \vee E_2) \wedge \cdots \wedge (B_m \vee E_m). \nonumber
\end{align}
Next, fix $I = (i_1,i_2,\ldots, i_k) \subset [m]$ and set $[m]-I = (j_i,j_2,\ldots j_{m-k})$. Then
$$\widehat{Z}\big(F_{0}K_{I};(\underline{C},\underline{E})_I\big)\wedge 
\widehat{Z}\big(K_{[m]- I};(\underline{B},\underline{B})_{[m]-I}\big)= (E_{i_1}\wedge E_{i_2}\wedge \cdots E_{i_k})\wedge
(B_{j_1}\wedge B_{j_2}\wedge \cdots B_{j_{m-k}}).$$
    \nd is the $I$-th wedge term in the expansion of the right hand side of \eqref{eqn:F0}. This confirms Theorem \ref{thm:main}
for $t=0$. 
\skp{0.3}
We suppose next the induction hypothesis that 
$$F_{t-1}\widehat{Z}(K;(\underline{U},\underline{V})) \;\simeq\;
\mvee_{I\leq [m]}\widehat{Z}\big(F_{t-1}K_{I};(\underline{C},\underline{E})_I\big)\wedge 
\widehat{Z}\big(K_{[m]- I};(\underline{B},\underline{B})_{[m]-I}\big),$$
\nd with a view to verifying it for $F_{t}$.  The definition of the filtration gives
{\fontsize{11}{12}\selectfont
\begin{align}\label{eqn:union}
F_{t}\widehat{Z}(K;(\underline{U},\underline{V})) \;&=\; 
\widehat{D}_{\underline{U},\underline{V}}(\sigma_{t}) \;\cup\; 
F_{t-1}\widehat{Z}(K;(\underline{U},\underline{V}))  \\
\nonumber &\simeq\;
\widehat{D}_{\underline{U},\underline{V}}(\sigma_{t}) \;\cup\; 
\mvee_{I\leq [m]}\widehat{Z}\big(F_{t-1}K_{I};(\underline{C},\underline{E})_I\big)\wedge 
\widehat{Z}\big(K_{[m]- I};(\underline{B},\underline{B})_{[m]-I}\big).
\end{align}}
\nd The space \;$\ds{\widehat{D}_{\underline{U},\underline{V}}(\sigma_{t})}$\; is the smash product
\begin{equation}\label{eqn:dsigma}
\mwedge^{m}_{j=1}B_{j}\vee Y_{j},\quad {\rm with}\quad
Y_{j}=\left\{\begin{array}{lcl}
C_{j}  &{\rm if} & j\in \sigma_{t}\\
E_{j}  &{\rm if} & j\notin \sigma_{t}.
\end{array}\right.
\end{equation}
\vspace{0.5mm}

\nd After a shuffle of wedge factors, the space
 $\ds{\widehat{D}_{\underline{U},\underline{V}}(\sigma_{t})}$ becomes
\begin{multline}\label{eqn:induction}
\mvee_{I\leq [m],\; \sigma_{t}\in I}
\widehat{D}_{\underline{C},\underline{E}}^{I}(\sigma_{t}) \wedge 
\widehat{Z}\big(K_{[m]- I};(\underline{B},\underline{B})_{[m]-I}\big)\;
\; \vee \\
\mvee_{I\leq [m],\; \sigma_{t}\notin I}
\widehat{Z}\big(K_{I};(\underline{E},\underline{E})_I\big) \wedge 
\widehat{Z}\big(K_{[m]- I};(\underline{B},\underline{B})_{[m]-I}\big)
\end{multline} 
\skp{0.2}
\nd where the space $\widehat{D}_{\underline{C},\underline{E}}^{I}(\sigma_{t})$ is defined by
\eqref{eqn:ced.sigma}. 
\skp{0.6}
\begin{rem}
Notice here the relevant fact that the number of subsets $I \leq [m]$ is
the same as the number of wedge summands in the expansion of \eqref{eqn:dsigma}, namely
$2^{m}$.
\end{rem}

\nd The right-hand wedge summand in \eqref{eqn:induction} is a subset of
$$\mvee_{I\leq [m]}\widehat{Z}\big(F_{t-1}K_{I};(\underline{C},\underline{E})_I\big)\wedge 
\widehat{Z}\big(K_{[m]- I};(\underline{B},\underline{B})_{[m]-I}\big)$$ 
\nd and so,
{\fontsize{9.7}{12}\selectfont
\begin{align*}
\mvee_{I\leq [m],\; \sigma_{t}\notin I}
\hspace{-2mm}\widehat{Z}\big(K_{I};(\underline{E},\underline{E})_I\big) \wedge 
\widehat{Z}\big(K_{[m]- I};(\underline{B},\underline{B})_{[m]-I}\big)\;\; &\mcup 
\mvee_{I\leq [m]}\widehat{Z}\big(F_{t-1}K_{I};(\underline{C},\underline{E})_I\big)\wedge 
\widehat{Z}\big(K_{[m]- I};(\underline{B},\underline{B})_{[m]-I}\big)\\[1mm]
&= \mvee_{I\leq [m]}\widehat{Z}\big(F_{t-1}K_{I};(\underline{C},\underline{E})_I\big)\wedge 
\widehat{Z}\big(K_{[m]- I};(\underline{B},\underline{B})_{[m]-I}\big).
\end{align*}}
Finally, for each 
${I\leq [m]}$ with ${\sigma_{t}\in I}$,  we have   

\begin{equation}
\widehat{D}_{\underline{C},\underline{E}}^{I}(\sigma_{t}) \cup 
\widehat{Z}\big(F_{t-1}K_{I};(\underline{C},\underline{E})_I\big) \;=\; 
\widehat{Z}\big(F_{t}K_{I};(\underline{C},\underline{E})_I\big).
\end{equation}
\skp{0.05}
\nd This concludes the inductive step to give
\begin{equation}\label{eqn:ftuv}
F_{t}\widehat{Z}(K;(\underline{U},\underline{V})) \;\simeq\;
\mvee_{I\leq [m]}\widehat{Z}\big(F_{t}K_{I};(\underline{C},\underline{E})_I\big)\wedge 
\widehat{Z}\big(K_{[m]- I};(\underline{B},\underline{B})_{[m]-I}\big).
\end{equation}
\nd It is straightforward to explicitly check the steps above in the case of $F_{0}$ and
$F_{1}$. This completes the proof. \hfill $\Box$

\section{\hspace{2mm}Cohomological wedge decomposability and the general case}\label{sec:cohomwdl}

The result of the previous section can be exploited to give information about 
the groups $\widetilde{H}^{\ast}\big(\widehat{Z}(K;(\underline{X},\underline{A}))\big)$ over a field $k$ for pointed, 
finite, path connected  pairs of 
CW-complexes $(\underline{X}, \underline{A})$ of finite type, which are {\bf {\em not\/}} wedge 
decomposable.

 \begin{defin}\label{defn:strongfc}
A {\em strongly free decomposition\/} of the homology of $(X,A)$ with coefficients in a ring $k$, is a 
quadruple of $k$-modules $(E', B', C', \overline{W})$ such that the long exact sequence
\begin{equation}\label{eqn:les}
\overset{\delta}{\to} \widetilde{H}^\ast(X/A) \overset{\ell}{\to} \widetilde{H}^\ast(X) \overset{\iota}{\to} \widetilde{H}^*(A)  
\overset{\delta}{\to} \widetilde{H}^{\ast+1}(X/A) \to  \end{equation}
\skp{0.1}
\nd satisfies the condition that there exist isomorphisms 
\nd \skp{0.2}
\begin{enumerate}\itemsep2mm
\item $\widetilde{H}^{\ast}(A)\cong B'\oplus E'$
\item  $\widetilde{H}^{\ast}(X)\cong B' \oplus C'$,  \; where $B' 
\underset{\simeq}{\overset{\iota}{\to}} B', \;\;\left.\iota\right|_{C'} = 0$ 
\item  $\widetilde{H}^{\ast}(X/A)\cong C' \oplus \overline{W}$, \;
where $C' \underset{\simeq}{\overset{\ell}{\to}} C', \;\; \left.\ell\right|_{B'} = 0, \;\; 
E' \underset{\simeq}{\overset{\delta}{\to}} \overline{W}$ 
\end{enumerate}
for free graded $k$-modules  $E', B', C'$ and $\overline{W}$ of finite type. 
\skp{0.2}
A pair $(X,A)$ with a strongly free decomposition is said to be {\it strongly free\/}. A morphism of strongly 
free pairs 
$$(f_\ell,f_\iota,f_\delta)\colon (X,A) \to (U,V)$$
is a morphism  of long exact sequences 
\small{\begin{equation}\label{eqn:wd}
\begin{tikzcd} 
\cdots  \arrow[r, "\delta^{(X,A)}"] &\widetilde{H}^*(X/A) \arrow[r, "\ell^{(X,A)}"], \arrow[d, "f_\delta"] 
&\widetilde{H}^\ast(X) \arrow[r, "\iota^{(X,A)}"] , \arrow[d, "f_\ell"] 
&\widetilde{H}^\ast(A)  \arrow[r, "\delta^{(X,A)}"], \arrow[d, "f_\iota"]  
&\widetilde{H}^{\ast+1}(X/A) \arrow[r, "l^{(X,A)}"] , \arrow[d, "f_\delta"] 
&\cdots\\
\cdots  \arrow[r, "\delta^{(U,V)}"] 
&\widetilde{H}^\ast(U/V) \arrow[r, "\ell^{(U,V)}"]  
&\widetilde{H}^\ast(U) 
\arrow[r, "\iota^{(U,V)}"] 
&\widetilde{H}^\ast(V)  \arrow[r, "\delta^{(U,V)}"] 
&\widetilde{H}^{\ast+1}(U/V) \arrow[r, "l^{(U,V)}"] 
&\cdots
\end{tikzcd}
\end{equation}}

\nd  with all diagrams commuting,  which restricts to maps of the submodules $E',B',C',\overline{W}$ 
corresponding to each pair.
\end{defin}
\begin{rem}
 When $k$ is a field, the homology of $(X,A)$ is always strongly free.   When $k = \mathbb{Z}$,
strong freeness holds if all the  spaces are torsion free and of finite type.
\end{rem}
Two lemmas about null homotopic inclusions are needed next. 
\begin{lem}\label{ecprops}
Let $\iota\colon E \hookrightarrow C$ be an inclusion of based CW complexes. then
\begin{enumerate}
\item There is a map $g\colon cE \longrightarrow C$ from the reduced cone,  so that the diagram 
below commutes
\begin{equation}\label{eqn:coneone}
\hspace{0.0cm}\begin{tikzcd}[column sep = 2cm]
C \vee cE 
\arrow[r, "1_C \;\vee\; g"] &C\\ 
E \arrow[r, "="] \arrow[u, "h"'] 
&E \arrow[u, "\iota"]
\end{tikzcd}
\end{equation}
where $h(e) = [(e,0)]$, that is, $h$ includes $E$ into the base of the cone.
\item There is a homotopy equivalence of colimits
\begin{equation}
 \widehat{Z}\big(\partial\overline{\sigma}_{t};(\underline{C \vee cE},\; \underline{E})\big) \longrightarrow
\widehat{Z}\big(\partial\overline{\sigma}_{t};(\underline{C},\; \underline{E})\big).
\end{equation}
\item The vertical maps in \eqref{eqn:coneone} have homotopy equivalent cofibers.
\end{enumerate}
\end{lem}
\begin{proof}
Since $\iota\colon E \hookrightarrow C$ is null homotopic, the exists a homotopy
$$G\colon E \times I \longrightarrow C$$
satisfying, for every $e\in E$, $G((e,0)) = \iota(e)$ and $G((e,1)) = e_0$, the base point of $E$,  which is also the
basepoint of $C$. 
This implies the existence of an extension
\begin{equation}\label{eqn:coemap}
\hspace{0.0cm}\begin{tikzcd}[column sep = 1cm]
E \times I 
\arrow[rr, "G"] \arrow[dr,"\pi"] &&C\\ 
&cE \arrow[ur,dashed, "g"] 
\end{tikzcd}
\end{equation}
\nd where $\pi$ projects $E\times \{1\}$ to the cone point. The commutativity of \eqref{eqn:coneone} follows.

Item (2)  follows from the Homotopy Lemma, (\cite[Lemma $1.7$]{zz}, because the horizontal maps in
\eqref{eqn:coneone} are homotopy equivalences and the commutativity implies that they give a map of diagrams
defining the colimits $\widehat{Z}\big(\partial\overline{\sigma}_{t};(\underline{C \vee cE},\; \underline{E})\big)$
and $\widehat{Z}\big(\partial\overline{\sigma}_{t};(\underline{C},\; \underline{E})\big)$.

Finally, the third item is the standard homotopy invariance of cofibers, see for example \cite[Theorem $2.3.7$]{mun}.
\end{proof}

Given a pair $(X,A)$ which is strongly free, let $B'$, $C'$ and $E'$ be the $k$--modules  
specified in items (1) and (2) of 
Definition \ref{defn:strongfc} for each pair $(X,A)$.
Now, wedges of spheres $B$, $C$ and $E$ exist with cohomology equal to the modules 
$B'$, $C'$ and $E'$ so that
\begin{equation}\label{eqn:wdpair}
(U,V) = (B\vee C, B \vee E)
 \end{equation}
\nd satisfies the criterion for a wedge decomposable pair as in Definition \ref{def:wedgedecomp}. 
 In particular, the inclusion  $\iota:E \to C$  is a null homotopic inclusion of wedges of spheres. 
Consider next diagram \eqref{eqn:wd} for the pairs $(X,A)$ and $(U,V)$, the latter defined as in \eqref{eqn:wdpair}.
 \begin{lem}\label{lem:xaisouv}
For a pair $(U,V)$ derived from $(X,A)$ in the manner above, there is an isomorphism 
$(f_\ell,f_\iota,f_\delta)\colon (X,A) \longrightarrow (U,V)$ 
of strongly free pairs  as in  \eqref{eqn:les} and \eqref{eqn:wd}.
\end{lem}
\begin{proof}
The map $f_l$ is the isomorphism 
$\widetilde{H}^\ast(X) \xrightarrow{\cong} B' \oplus C' = \widetilde{H}^\ast(U)$
and the map $f_\iota$  is the compatible isomorphism 
$\widetilde{H}^\ast(A) \xrightarrow{\cong} B' \oplus E' = \widetilde{H}^\ast(V)$, both from Definition \ref{defn:strongfc}. 
 Since the inclusion $\iota:E \to C$  is null homotopic, we apply Lemma \ref{ecprops} to write it as 
$E \hookrightarrow cE\vee C$ where $cE$ is the  unreduced  cone and the inclusion is onto the base of the cone. 
 Again from Lemma \ref{ecprops}, item (3), it follows that $C/E \;\simeq\; \Sigma{E}\vee C$
and hence $U/V \simeq  \Sigma{E}\vee C$. This gives the isomorphism 
$$f_\delta\colon \widetilde{H}^\ast(X/A) \xrightarrow{\cong} \overline{W}\oplus C'  \xrightarrow{=}
\widetilde{H}^\ast(\Sigma{E}\vee C) \xrightarrow{\cong} \widetilde{H}^\ast(U/V).$$ 
Finally, it follows from the definitions that the diagrams \eqref{eqn:wd} all commute.\end{proof}

 Our goal is to show that under the strong freeness condition of Definition \ref{defn:strongfc},  
an analogue of Theorem \ref{thm:main} holds for pairs $(\underline{X}, \underline{A})$, using  
pairs  $(\underline{U},\underline{V})$ where each $(U_i,V_i)$ has been constructed from $(X_i,A_i)$  via 
\eqref{eqn:wdpair} above.

\begin{thm}\label{thm:cartan2}
Under the conditions stated above, there is an isomorphism of cohomology groups with coefficients
in a field $k$ 
$$\theta_{(\underline{U},\underline{V})}\colon \widetilde{H}^{*}\big(\widehat{Z}(K;(\underline{U},\underline{V}))\big) \longrightarrow
 \widetilde{H}^{*}\big(\widehat{Z}(K;(\underline{X},\underline{A}))\big)$$ 
\nd where the left hand side is determined by  Corollary \ref{cor:wedge}.
 (This is not necessarily an isomorphism of modules over the Steenrod algebra and does not preserve products 
 in general.)
\end{thm} 
\begin{cor}\label{cor:equiv }
Let $(\underline{X},\underline{A})$ and $(\underline{Y},\underline{B})$ be two families of strongly
free pairs so that there is an isomorphism $(f_\ell,f_\iota,f_\delta)$ 
of strongly free pairs $(X_i,A_i) \to (Y_i,B_i)$ for each $i\in [m]$ as in  \eqref{eqn:les} and \eqref{eqn:wd},
then there is an isomorphism of groups for cohomology with coefficients in a field $k$ 
$$\widetilde{H}^{*}\big(\widehat{Z}(K;(\underline{X},\underline{A}))\big) \longrightarrow
 \widetilde{H}^{*}\big(\widehat{Z}(K;(\underline{Y},\underline{B}))\big).$$ 
\end{cor} 
\begin{proof}
For each $i\in [m]$, the associated wedge decomposable pairs $(U_i,V_i)$ given by \eqref{eqn:wdpair}
for both $(\underline{X},\underline{A})$ and $(\underline{Y},\underline{B})$ are the same and so the result
follows from Theorem \ref{thm:cartan2}. \end{proof}
\skp{0.1}
The proof of Theorem \ref{thm:cartan2} is centered around a result of
J.~Grbic and S.~Theriault \cite{gt}, as described in \cite[Section 3]{bbcg10}.
 In order to keep track of ghost vertices, we introduce further notation.
For \mbox{$\sigma = \{i_1,i_2,\ldots,i_{n+1}\}$}, with complementary vertices $\{j_1,j_2,\ldots,j_{m-n-1}\}$, set
\begin{equation}\label{eqn:ghost}
\widehat{A}^{[m]-\sigma} = A_{j_1}\wedge A_{j_2}\wedge \cdots\wedge A_{j_{m-n-1}}
\end{equation}
\begin{thm}\cite{gt, bbcg10}\label{thm:gtseq}
For general pairs $(\underline{X},\underline{A})$, the diagram below is commutative diagram of
cofibrations for each $t$, $0\leq t\leq s$.
\small{\begin{equation}\label{eqn:gt}
\begin{tikzcd}
\cdots  \arrow[r, "\delta^{(X,A)}_t"] & F_{t-1}\widehat{Z}(K;(\underline{X},\underline{A}))
\arrow[r, "\iota"] 
 &F_{t}\widehat{Z}(K;(\underline{X},\underline{A})) 
 \arrow[r, "\gamma^{(X,A)}_t"] &  \mathcal{C}_{(\underline{X},\underline{A})} \cdots \\
 \cdots  \arrow[r, "\delta^{(U,V)}_ {\sigma_t}"] &\widehat{Z}(\partial\overline{\sigma}_{t};(\underline{X},\underline{A}))
 \wedge \widehat{A}^{[m]-\sigma}
 \arrow[u, "g_{\partial\sigma_t}"'] \arrow[r, "\iota"]
&\widehat{Z}\overline{\sigma}_{t};(\underline{X},\underline{A}))\arrow[r,"\gamma^{(U,V)}_{\sigma_t} "]  
\wedge \widehat{A}^{[m]-\sigma} 
\arrow[u, "g_{\sigma_t}"]
& \mathcal{C}_{(\underline{X},\underline{A})} \arrow[u, "=", "c_{t}"']\cdots
\end{tikzcd}
\end{equation}}
\nd where here, $\overline{\sigma}_{t}$ and $\partial\overline{\sigma}_{t}$ represents the simplex $\sigma_t$ and 
the boundary of $\overline{\sigma}_{t}$ respectively, considered  without ghost vertices, that is, as simplicial complexes
on $ \{i_1,i_2,\ldots,i_{n+1}\}$, the vertices of the simplex only.
\end{thm} 
\skp{0.2}
An outline of the proof of Theorem \ref{thm:cartan2} is as follows.
\begin{enumerate}\itemsep2mm
\item The lower cofibration in \eqref{eqn:gt} is analyzed geometrically in the case that
$(\underline{X},\underline{A})$  is a wedge decomposable pair $(\underline{U},\underline{V})$. This is done in
Section \ref{sec:wdcofibl}.
\item  In Section \ref{sec:boundary} these results are used then to prove Theorem \ref{thm:cartan2} for $K = \partial{\sigma_t}$, the 
boundary of a simplex.
\item An inductive argument in Section \ref{sec:general} uses Diagram \eqref{eqn:gt} to complete the proof for general $K$.
\end{enumerate}
\section{\hspace{2mm}The canonical cofibration for wedge decomposable pairs}\label{sec:wdcofibl}
The three spaces in the lower cofibration of \eqref{eqn:gt} are now analyzed in the case
of a wedge-decomposable pair 
$(\underline{U},\underline{V}) = (\underline{B\vee C},\underline{B\vee E})$.  For $\sigma_{t}$ an 
$n$-simplex on vertices $\{i_{1},i_{2},\ldots,i_{n+1}\}$, with no ghost vertices, the identification of the
space $\widehat{Z}(\overline{\sigma}_{t};(\underline{U},\underline{V}))$ is straightforward.
\begin{equation}\label{eqn:candd}
\widehat{Z}(\overline{\sigma}_{t};(\underline{U},\underline{V})) \;=\;
(B_{i_{1}} \vee C_{i_{1}})\wedge (B_{i_{2}} \vee C_{i_{2}})\wedge \cdots\wedge (B_{i_{n+1}} \vee C_{i_{n+1}})
\end{equation}
$$\simeq\; C_{i_{1}} \wedge C_{i_{2}} \wedge \cdots \wedge C_{i_{n+1}} \vee \overline{D}_{\sigma_t}$$
\nd where $\overline{D}_{\sigma_t}$ is a wedge of smash products of spaces $B_{i_{j}}$ and $C_{i_{k}}$.
\skp{0.2}
For general $(\underline{X},\underline{A})$,  the short argument in \cite[Lemma 3.6]{bbcg10}
shows that  the cofiber in Theorem \ref{thm:gtseq}  is given by 
\begin{equation}\label{eqn:cofiberxa}
\mathcal{C}_{(\underline{X},\underline{A})} \simeq 
X_{i_{1}}/A_{i_{1}}\wedge X_{i_{2}}/A_{i_{2}}\wedge \ldots \wedge X_{i_{n+1}}/A_{i_{n+1}} \wedge \widehat{A}^{[m]-\sigma}
\end{equation}
\nd which we write as
\begin{equation}\label{eqn:cofiberbar}
\mathcal{C}_{(\underline{X},\underline{A})} \simeq \overline{\mathcal{C}}_{(\underline{X},\underline{A})}
\wedge \widehat{A}^{[m]-\sigma} 
\end{equation}

 Once again, for 
$(\underline{X},\underline{A}) = (\underline{U},\underline{V}) =(\underline{B\vee C},\underline{B\vee E})$,
we use Lemma \ref{ecprops} to replace the inclusion $E_{i_{j}} \hookrightarrow C_{i_{j}}$ with
$E_{i_{j}} \hookrightarrow cE_{i_{j}}\vee C_{i_{j}}$ where $cE_{i_{j}}$ is the cone. Again, we have $C_{i_{j}}/E_{i_{j}} 
\;\simeq\; \Sigma{E_{i_{j}}}\vee C_{i_{j}}$ and we get the next lemma. 

\begin{lem}\label{lem:cuv}
The cofiber  $\overline{\mathcal{C}}_{(\underline{U},\underline{V})}$, \eqref{eqn:cofiberbar}, for a wedge decomposable 
family of pairs, decomposes homotopically into a wedge of spaces as follows.
\begin{align*}\label{eqn:cuv}
\begin{split}
 \overline{\mathcal{C}}_{(\underline{U},\underline{V})} \;&\simeq\; (\Sigma{E_{i_{1}}}\vee C_{i_{1}}) 
\wedge (\Sigma{E_{i_{2}}}\vee C_{i_{2}}) \wedge
\ldots \wedge (\Sigma{E_{i_{n+1}}}\vee C_{i_{n+1}})\\
&\simeq\; \Sigma\big({E_{i_{1}}} \ast {E_{i_{2}}} \ast \cdots \ast {E_{i_{n+1}}}\big)\\ 
&\hspace{0.25in}\vee\;  \mvee_{j=1}^{n+1}\;\; C_{i_{j}.} \wedge \Sigma\big({E_{i_{1}}} \ast {E_{i_{2}}} \ast 
\cdots \ast \widehat{E}_{i_{j}} \ast \cdots  \ast {E_{i_{n+1}}}\big)\\
&\hspace{0.25in}\vee\; \mvee_{k_{1}<k_{2}}^{n+1}\;\; C_{i_{k_{1}}} \wedge C_{i_{k_{2}}}  \wedge 
\Sigma\big({E_{i_{1}}} \ast {E_{i_{2}}} \ast 
\cdots \ast \widehat{E}_{i_{k_{1}}} \ast \cdots \ast \widehat{E}_{i_{k_{2}}} \ast
\cdots  \ast {E_{i_{n+1}}}\big)\\
&\hspace{0.25in}\vee\; \cdots\\
&\hspace{0.25in}\vee\;  \mvee_{j = 1}^{n+1}\Sigma{E_{i_{j}}} \wedge C_{i_{1}} \wedge C_{i_{2}} \wedge \ldots  
\wedge \widehat{C}_{i_{j}}
\wedge \ldots \wedge C_{i_{n+1}} \\
&\hspace{0.25in}\vee\; C_{i_{1}} \wedge C_{i_{2}}  \wedge \ldots \wedge C_{i_{n+1}} 
\end{split}
\end{align*}
\end{lem}
\nd where here,  the equivalence $\Sigma(S \wedge T) \simeq S \ast T$ has been used iteratively.
On the other hand, for the simplicial complex $\partial{\overline{\sigma}_t}$ we have the following lemma.
\begin{lem}\label{lem:boundaryuv}
The polyhedral smash product 
$\widehat{Z}\big(\partial\overline{\sigma}_{t};(\underline{C \vee cE}, \underline{E})\big)$ decomposes
homotopically into a wedge os spaces as follows.
\skp{0.15}
\nd \phantom{mm}$\ds{\widehat{Z}\big(\partial\overline{\sigma}_{t};(\underline{C},\; \underline{E})\big)
\simeq \hspace{0.065in} \widehat{Z}\big(\partial\overline{\sigma}_{t};(\underline{C \vee cE},\; \underline{E})\big)}$
\begin{align*}
&\simeq \hspace{0.065in} \mcup_{k=1}^{n+1}\; (C_{i_{1}}\vee cE_{i_{1}}) \wedge\ldots \wedge
(C_{i_{k-1}}\vee cE_{i_{k-1}})  \wedge E_{i_{k}} \wedge  (C_{i_{k+1}}\vee cE_{i_{k+1}}) \wedge \ldots \wedge
(C_{i_{n+1}}\vee cE_{i_{n+1}}) \\
\\
&\simeq\hspace{0.065in} {E_{i_{1}}} \ast {E_{i_{2}}} \ast \cdots \ast {E_{i_{n+1}}}\\ 
&\hspace{0.25in}\vee\;  \mvee_{j=1}^{n+1}\;\; C_{i_{j}} \wedge {E_{i_{1}}} \ast {E_{i_{2}}} \ast 
\cdots \ast \widehat{E}_{i_{j}} \ast \cdots  \ast {E_{i_{n+1}}}\\
&\hspace{0.25in}\vee  \mvee_{k_{1}<k_{2}}^{n+1}\;\; C_{i_{k_{1}}} \wedge C_{i_{k_{2}}}  \wedge 
\big({E_{i_{1}}} \ast {E_{i_{2}}} \ast 
\cdots \ast \widehat{E}_{i_{k_{1}}} \ast \cdots \ast \widehat{E}_{i_{k_{2}}} \ast
\cdots  \ast {E_{i_{n+1}}}\big)\\
&\hspace{0.25in}\vee\;   \cdots\\
&\hspace{0.25in}\vee\;  \mvee_{j = 1}^{n+1}{E_{i_{j}}} \wedge C_{i_{1}} \wedge C_{i_{2}} \wedge \ldots  
\wedge \widehat{C}_{i_{j}}
\wedge \ldots \wedge C_{i_{n+1}} \\
\end{align*}
\end{lem}
\skp{-0.8}
\nd where this time, the equivalence 
$(Y \wedge cY)\mcup_{Y\wedge Y} (cY\wedge Y) \simeq Y\ast Y$ as been used iteratively.

\begin{rem}
This decomposition is the same as that given by the Wedge Lemma, 
\cite[Theorem $2.12$]{bbcg1}, in the case $K =  \partial\sigma$, the boundary of a simplex.
\end{rem}

 Finally, comparing the two decompositions above, we arrive at the identity
\begin{equation}\label{eqn:cofibre}
\overline{\mathcal{C}}_{(\underline{U},\underline{V})}\;\simeq\; 
\Sigma\widehat{Z}\big(\partial\overline{\sigma}_t;(\underline{C},\; \underline{E})\big) \;\vee\;
\big(C_{i_{1}} \wedge C_{i_{2}}  \wedge \ldots \wedge C_{i_{n+1}}\big). 
\end{equation}

\begin{defin}\label{defn:ghost}
We introduce now notational abbreviations which account for ghost vertices.  Here,
$(\underline{U}, \underline{V}) = (\underline{B\vee C}, \underline{B\vee E})$  and $\overline{\sigma}_t$,
$\partial\overline{\sigma}_t$ are as in Theorem \ref{thm:gtseq}.
\begin{enumerate}
\item $\widehat{Z}\big(\sigma_t;(\underline{Y}, \underline{Q})\big) = 
\widehat{Z}\big(\overline{\sigma}_t;(\underline{Y}, \underline{Q})\big) \wedge \widehat{Q}^{[m]-\sigma_t}$\;
for any family of pairs $(\underline{Y}, \underline{Q})$.
\item $\widehat{Z}\big(\partial\sigma_t;(\underline{Y}, \underline{Q})\big) = 
\widehat{Z}\big(\partial\overline{\sigma}_t;(\underline{Y}, \underline{Q})\big) \wedge \widehat{Q}^{[m]-\sigma_t}$\;
for any family of pairs $(\underline{Y}, \underline{Q})$.
\item $E_{{\sigma}_{t}} = \widehat{Z}\big(\partial\overline{\sigma}_t;(\underline{C}, \underline{E})\big) \wedge \widehat{V}^{[m]-\sigma_t}$
\item $C_{\sigma_t} = C_{i_{1}} \wedge C_{i_{2}}  \wedge \ldots \wedge C_{i_{n+1}} \wedge \widehat{V}^{[m]-\sigma_t}$
\item $\mathcal{C}_{(\underline{Y},\underline{Q})} = \overline{\mathcal{C}}_{(\underline{Y},\underline{Q})} 
\wedge \widehat{Q}^{[m]-\sigma_t}$\;
for any family of pairs $(\underline{Y}, \underline{Q})$.
\item $D_{\sigma_t} = \overline{D}_{\sigma_t} \wedge \widehat{V}^{[m]-\sigma_t}$, (see \eqref{eqn:candd}).
\end{enumerate}
\end{defin}
\skp{0.2} 
\begin{thm}\label{thm:wdsplitting}
\nd For the  pair $(\underline{U}, \underline{V}) = (\underline{B\vee C}, \underline{B\vee E})$,  the lower
sequence corresponding to \eqref{eqn:gt} 
\small{\begin{equation*}\label{eqn:split}
\longrightarrow \widehat{Z}\big(\partial\sigma_t;(\underline{U}, \underline{V})\big)  \xrightarrow{i}
\widehat{Z}\big(\sigma_t;(\underline{U}, \underline{V})\big) \
\xrightarrow{\gamma^{(U,V)}_{\sigma_t}} \mathcal{C}_{(\underline{U},\underline{V})}
\xrightarrow{\delta^{(U,V)}_{\sigma_t}}
\Sigma\widehat{Z}\big(\partial\sigma_t;(\underline{U}, \underline{V})\big) \xrightarrow{i}
\end{equation*}}
\nd splits geometrically as: 
\begin{equation}\label{eqn:splitsequence}
\longrightarrow E_{\sigma_t} \vee D_{\sigma_t}  \xrightarrow{i}
C_{\sigma_t} \vee D_{\sigma_t}  
\xrightarrow{\gamma^{(U,V)}_{\sigma_t}} \Sigma{E_{\sigma_t}}\;\vee\; C_{\sigma_t} \xrightarrow{\delta^{(U,V)}_{\sigma_t}}
\Sigma{E_{\sigma_t}} \vee \Sigma{D_{\sigma_t}}  \xrightarrow{i}
\end{equation}
\nd where the function 
$i$ maps $D_{\sigma_t}$ by the identity and the function $\gamma^{(U,V)}_{t}$ maps $C_{\sigma_t}$ by the identity.
\end{thm}

\section{\hspace{2mm}The proof of Theorem \ref{thm:cartan2} for the boundary of a simplex}\label{sec:boundary}
Returning now to \eqref{eqn:gt}, we begin to examine the consequences of Lemma \ref{lem:xaisouv}
 \begin{lem}\label{lem:sigmaiso}
There is an isomorphism of groups
\begin{equation}\label{eqn:sigmaiso}
\widetilde{H}^{*}\big(\widehat{Z}\big(\sigma_t;(\underline{X}, \underline{A})\big)\big) 
\xrightarrow{\cong} \widetilde{H}^{*}\big(\widehat{Z}\big(\sigma_t;(\underline{U}, \underline{V})\big)\big).
\end{equation}
\end{lem}
\nd {\em Proof\/}.
Applying Lemma \ref{lem:xaisouv} to Definition \ref{defn:ghost} part (3), and using the K\"unneth Theorem,
we have
\begin{align*}
\widetilde{H}^{*}\big(\widehat{Z}\big({\sigma}_t;(\underline{X}, \underline{A})\big)\big) &\;\cong\; 
\widetilde{H}^{*}\big(\widehat{Z}\big(\overline{\sigma}_t;(\underline{X}, \underline{A})\big)\big) \otimes
\widetilde{H}^{*}(\widehat{A}^{[m]-\sigma})\\
&\;\cong\; \widetilde{H}^{*}(X_{i_{1}}\wedge X_{i_{2}}\wedge \cdots \wedge X_{i_{n+1}})\otimes
\widetilde{H}^{*}(A_{j_{1}}\wedge A_{j_{2}}\wedge \cdots \wedge A_{j_{m-n-1}})\\
&\;\cong\; \widetilde{H}^{*}(X_{i_{1}})\otimes \cdots \otimes \widetilde{H}^{*}(X_{i_{n+1}}) \otimes
\widetilde{H}^{*}(A_{j_{1}})\otimes \cdots \otimes \widetilde{H}^{*}(A_{j_{m-n-1}})\\
&\;\cong\; \widetilde{H}^{*}(U_{i_{1}})\otimes  \cdots \otimes \widetilde{H}^{*}(U_{i_{n+1}})\otimes
\widetilde{H}^{*}(V_{j_{1}})\otimes \cdots \otimes \widetilde{H}^{*}(V_{j_{m-n-1}})\\
&\;=\; \widetilde{H}^{*}\big(\widehat{Z}\big({\sigma}_t;(\underline{U}, \underline{V})\big)\big)
\hspace{3.16truein}\square
\end{align*}  
\nd The strong freeness condition and  Lemma \ref{lem:xaisouv} yield the next lemma in an analogous way. 
\begin{lem}\label{lem:cofibreiso}
There is an isomorphism of groups
\begin{equation}\label{eqn:cofibers}
\widetilde{H}^{*}(\mathcal{C}_{(\underline{X},\underline{A})}) \xrightarrow{\cong} \widetilde{H}^{*}(\mathcal{C}_{(\underline{U},\underline{V})}),
\end{equation}
\end{lem}
\nd {\em Proof\/}. Lemma \ref{lem:xaisouv}, \eqref{eqn:cofiberxa} and the K\"unneth Theorem give
\begin{align*}
\widetilde{H}^{*}(\mathcal{C}_{(\underline{X},\underline{A})}) 
&\;\cong\; \widetilde{H}^{*}(X_{i_{1}}/A_{i_{1}} \wedge X_{i_{2}}/A_{i_{2}} \wedge \cdots \wedge X_{i_{n+1}}/A_{i_{n+1}}) \otimes
\widetilde{H}^{*}(A_{j_{1}}\wedge A_{j_{2}}\wedge \cdots \wedge A_{j_{m-n-1}})\\
&\;\cong\; \widetilde{H}^{*}(X_{i_{1}}/A_{i_{1}}) \otimes\cdots\otimes \widetilde{H}^{*}(X_{i_{n+1}}/A_{i_{n+1}})\otimes 
\widetilde{H}^{*}(A_{j_{1}})\otimes \cdots \otimes \widetilde{H}^{*}(A_{j_{m-n-1}})\\
&\;\cong\; \widetilde{H}^{*}(U_{i_{1}}/V_{i_{1}}) \otimes\cdots\otimes \widetilde{H}^{*}(U_{i_{n+1}}/V_{i_{n+1}})\otimes 
\widetilde{H}^{*}(V_{j_{1}})\otimes \cdots \otimes \widetilde{H}^{*}(V_{j_{m-n-1}})\\
&\;\cong\; \widetilde{H}^{*}(\mathcal{C}_{(\underline{U},\underline{V})}) \hspace{4.35truein}\square
\end{align*}   
\nd The next lemma extends the isomorphism \eqref{eqn:sigmaiso} to the boundary of a simplex.
\begin{lem}\label{lem:boundaryiso}
There is an isomorphism of groups
$$\phi_t\colon \widetilde{H}^{\ast}\big(\widehat{Z}\big(\partial\sigma_t;(\underline{U}, \underline{V})\big)\big) 
= \widetilde{H}^{\ast}\big(E_{\sigma_{t}}\big) \oplus \widetilde{H}^{\ast}\big(D_{\sigma_{t}}\big)
\longrightarrow \widetilde{H}^{\ast}\big(\widehat{Z}\big(\partial\sigma_t;(\underline{X}, \underline{A})\big)\big).$$
\end{lem}

\begin{proof}
\nd Consider the ladder arising from the lower cofibration in \eqref{eqn:gt}, for both $(\underline{X}, \underline{A})$
and  $(\underline{U}, \underline{A})$. We have adopted the notation of \eqref{eqn:splitsequence} for the latter.

\hspace{-12mm}{\footnotesize \begin{tikzcd}
 &\cdots \widetilde{H}^{\ast}\big(\widehat{Z}\big(\sigma_t;(\underline{X}, \underline{A})\big)\big) 
\arrow[r, "\iota"] 
&\widetilde{H}^{\ast}\big(\widehat{Z}\big(\partial\sigma_t;(\underline{X}, \underline{A})\big)\big) 
 \arrow[r, "\delta^{(X,A)}_{\sigma_t}"] 
&\widetilde{H}^{\ast+1}(\mathcal{C}_{(\underline{X},\underline{A}}) \cdots \\
&\cdots \widetilde{H}^{\ast}\big(C_{\sigma_t}\big) \oplus \widetilde{H}^{\ast}\big(D_{\sigma_t}\big)
\arrow[u, "\beta_t"',"\cong"] \arrow[r, "\iota"]
&\widetilde{H}^{\ast}\big(E_{\sigma_{t}}\big) \oplus \widetilde{H}^{\ast}\big(D_{\sigma_{t}}\big) \arrow[r,"\delta^{(U,V)}_{\sigma_t} "]  
\arrow[u, dashed,"\phi_t"]
&\widetilde{H}^{\ast+1}\big(\Sigma{E_{\sigma_t}}_{\sigma_{t}}\big) \oplus \widetilde{H}^{\ast+1}\big(C_{\sigma_{t}}\big) 
\arrow[u, "\kappa_t"',"\cong"]  \cdots
\end{tikzcd}}

We use next the geometric splitting from \eqref{eqn:splitsequence} and the vertical isomorphisms to define a map 
$$\phi_t\colon \widetilde{H}^{\ast}\big(\widehat{Z}\big(\partial\sigma_t;(\underline{U}, \underline{V})\big)\big) 
= \widetilde{H}^{\ast}\big(E_{\sigma_{t}}\big) \oplus \widetilde{H}^{\ast}\big(D_{\sigma_{t}}\big)
\longrightarrow \widetilde{H}^{\ast}\big(\widehat{Z}\big(\partial\sigma_t;(\underline{X}, \underline{A})\big)\big).$$
\nd  Theorem \ref{thm:wdsplitting} gives 
$$\widetilde{H}^{\ast}\big( \widehat{Z}\big(\partial\sigma_t;(\underline{U}, \underline{V})\big)\big)\;\cong\;
\widetilde{H}^{\ast}\big(C_{\sigma_t}\big) \oplus \widetilde{H}^{\ast}(D_{\sigma_t}),$$
and so, working over a field, we consider the splitting
$$\widetilde{H}^{\ast}\big(\widehat{Z}\big(\partial\sigma_t;(\underline{X}, \underline{A})\big)\big) \;\cong\;
\widetilde{H}^{\ast}\big(\widehat{Z}\big(\partial\sigma_t;(\underline{X}, \underline{A})\big)\big)
\big/(\iota\circ \beta_{t})\big(\widetilde{H}^{\ast}(D_{\sigma_t})\big)\; \oplus\; (\iota\circ \beta_{t})\big(\widetilde{H}^{\ast}(D_{\sigma_t})\big). $$ 
For  $d \in \widetilde{H}^{\ast}(D_{\sigma_t})$, set $\phi_{t}(d)  = (\iota\circ \beta_{t})(d)$, and for 
$e \in \widetilde{H}^{\ast}\big(E_{\sigma_t}\big)$, set $\phi_{t}(e)$  
equal to the unique
class $u$ in  
$$\widetilde{H}^{\ast}\big(\widehat{Z}\big(\partial\sigma_t;(\underline{X}, \underline{A})\big)\big)
\big/ (\iota\circ \beta_{t})\big(\widetilde{H}^{\ast}(D_{\sigma_t})\big)$$
such that $\delta^{(X,A)}_{\sigma_t}(u) = \big(\kappa_{t}\circ \delta^{(U,V)}_{\sigma_t}\big)(e)$.  
The diagram commutes by the construction of the map $\phi_t$.
The Five-Lemma implies now that the map $\phi_t$ is an isomorphism. 
\end{proof}

\section{\hspace{2mm}The proof of Theorem \ref{thm:cartan2} for general $K$}\label{sec:general}
\skp{0.2}
We consider below a diagram of vector spaces over a field. It is constructed from the commutative
diagram \eqref{eqn:gt} applied to the pairs $(\underline{U},\underline{V})$ and $(\underline{X},\underline{A})$ and
incorporating the isomorphisms from Lemma \ref{lem:boundaryiso}. We assume by way of induction that
$$\widetilde{H}^{\ast}\big(F_{t-1}\widehat{Z}\big(K;(\underline{U}, \underline{V})\big)\big) \cong
\widetilde{H}^{\ast}\big(F_{t-1}\widehat{Z}\big(K;(\underline{X}, \underline{A})\big)\big),$$ which  is true when $F_{t-1}K$ is a simplex by 
\eqref{eqn:flitration} and Theorem \ref{eqn:sigmaiso}.

\skp{0.2}
{\footnotesize \hspace{-12mm}\begin{tikzcd}
\cdots \widetilde{H}^{\ast}(\mathcal{C}_{(\underline{U},\underline{V}}) 
\arrow[r, "\gamma^{(U,V)}_t"]  \arrow[d, "c_t"',"\cong"] 
 &\widetilde{H}^{\ast}\big(F_t\widehat{Z}\big(K;(\underline{U}, \underline{V})\big)\big) 
 \arrow[r, "\iota"]  \arrow[d, "g_{\sigma_t}"']
 & \widetilde{H}^{\ast}\big(F_{t-1}\widehat{Z}\big(K;(\underline{U}, \underline{V})\big)\big) 
\arrow[d, "g_{\partial{\sigma_t}}"]  \arrow[r, "\delta^{(U,V)}_t"]
& \widetilde{H}^{\ast+1}(\mathcal{C}_{(\underline{U},\underline{V}})  \arrow[d, "c_t"',"\cong"]  \cdots\\ 
\cdots \widetilde{H}^{\ast}\big(\Sigma{E_{\sigma_t}}\big) \oplus \widetilde{H}^{\ast}\big(C_{\sigma_{t}}\big)
 \arrow[r, "\gamma^{(U,V)}_{\sigma_t}"]
&\widetilde{H}^{\ast}\big(C_{\sigma_t}\big) \oplus \widetilde{H}^{\ast}\big(D_{\sigma_t}\big) \arrow[r,"\iota"] 
&\widetilde{H}^{\ast}\big(E_{\sigma_{t}}\big) \oplus \widetilde{H}^{\ast}\big(D_{\sigma_{t}}\big) \arrow[r, "\delta^{(U,V)}_{\sigma_t}"]
&\widetilde{H}^{\ast+1}\big(\Sigma{E_{\sigma_t}}\big) \oplus \widetilde{H}^{\ast}\big(C_{\sigma_{t}}\big) \cdots \\
\cdots \widetilde{H}^{\ast}(\mathcal{C}_{(\underline{X},\underline{A}}) 
\arrow[r, "\gamma^{(X,A)}_t"]  \arrow[u, "c_t","\cong"']
 &\widetilde{H}^{\ast}\big(F_t\widehat{Z}\big(K;(\underline{X}, \underline{A})\big)\big) 
 \arrow[r, "\iota"]  \arrow[u, "g_{\sigma_t}"]
 &\widetilde{H}^{\ast}\big(F_{t-1}\widehat{Z}\big(K;(\underline{X}, \underline{A})\big)\big)  
  \arrow[u, "g_{\partial{\sigma_t}}"']  \arrow[r, "\delta^{(X,A)}_t"] 
    &\widetilde{H}^{\ast+1}(\mathcal{C}_{(\underline{X},\underline{A}})   \arrow[u, "c_t","\cong"']\cdots
\end{tikzcd}}
\skp{0.2}

\nd The exactness and the commutativity of the diagram implies that we can choose isomorphisms as follows

\begin{enumerate}\itemsep3mm
\item[] $\widetilde{H}^{\ast}\big(F_t\widehat{Z}\big(K;(\underline{U}, \underline{V})\big)\big) \cong \widetilde{H}^{\ast}\big(C_{\sigma_t}\big) \oplus L$ \; for some $L$
\item[] $\widetilde{H}^{\ast}\big(F_{t-1}\widehat{Z}\big(K;(\underline{U}, \underline{V})\big)\big) \cong \widetilde{H}^{\ast+1}\big(\Sigma{E_{\sigma_t}}\big) \oplus L$
\item[] $\widetilde{H}^{\ast}\big(F_t\widehat{Z}\big(K;(\underline{X}, \underline{A})\big)\big)  \cong \widetilde{H}^{\ast}\big(C_{\sigma_t}\big) \oplus L'$ \; for some $L'$
\item[] $\widetilde{H}^{\ast}\big(F_{t-1}\widehat{Z}\big(K;(\underline{X}, \underline{A})\big)\big) \cong \widetilde{H}^{\ast+1}\big(\Sigma{E_{\sigma_t}}\big) \oplus L'$
\end{enumerate}
\skp{0.1}
\nd The inductive hypothesis  implies now that $L \cong L'$ and so
$$\widetilde{H}^{\ast}\big(F_t\widehat{Z}\big(K;(\underline{U}, \underline{V})\big)\big) \cong 
\widetilde{H}^{\ast}\big(F_t\widehat{Z}\big(K;(\underline{X}, \underline{A})\big)\big)$$ 
as required. This, together with the fact that result is true for a simplex and its boundary, (Section \ref{sec:boundary}),
completes the proof.

\section{\hspace{0.3mm} The Hilbert-Poincar\'e series for $Z(K;(X,A))$}\label{sec:hpseries}
We begin by reviewing some of the elementary properties of Hilbert-Poincar\'e series. Assume now that homology is taken with
coefficients in a field $k$ and all spaces are pointed, path conected with the homotopy type of CW-complexes. The Hilbert-Poincar\'e series
$$P(X,t) = \sum_{n}\big({\rm dim}_{k}H_{n}(X;k)\big)t^n$$
\nd and the reduced Hilbert-Poincar\'e series
$$ \overline{P}(X,t) = -1 + P(X,t)$$
\nd satisfy the following properties.
\begin{enumerate}
\item $P(X,t)P(Y,t) = P(X\times Y),t)$, and
\item  $\overline{P}(X,t) \overline{P}(Y,t)  = P(X\wedge Y,t)$.
\end{enumerate}
For a pair $(X,A)$ satisfying the conditions of Theorem \ref{thm:cartan2}, we have 
\begin{equation}
\overline{P}\big(\widehat{Z}(K;(\underline{X},\underline{A})),t\big) = \overline{P}\big(\widehat{Z}(K;(\underline{U},\underline{V})),t\big)
\end{equation}
where the pair $(U,V)$ is as in Definition \ref{defn:strongfc}. Next, Theorem \ref{thm:main} gives
\begin{equation}
\overline{P}\big(\widehat{Z}(K;(\underline{U},\underline{V})),t\big) =
\sum_{I\leq [m]} \Big[\overline{P}\big(\widehat{Z}(K_I;(\underline{C},\underline{E})_I),t\big)\big) 
\cdot\lprod_{j\in [m]-I}\overline{P}(B_{i},t)\Big]
\end{equation}
We apply now Corollary \ref{cor:wedge} to refine this further and obtain the next theorem.
\begin{thm}
The reduced Hilbert-Poincar\'e series for $\widehat{Z}(K;(\underline{U},\underline{V}))$, and hence for
$\widehat{Z}(K;(\underline{X},\underline{A}))$ is given as follows,
\begin{equation*}\label{eqn:hpszhat}
\overline{P}\big(\widehat{Z}(K;(\underline{U},\underline{V})),t\big) =
\sum_{I\leq [m]}\Big[\sum_{\sigma \in K_I}\big[(t)\overline{P}(|\text{lk}_{\sigma}(K_{I})|,t)\cdot
\overline{P}\big(\widehat{D}_{\underline{C},\underline{E}}^{I}(\sigma)\big)\big]\cdot\lprod_{j\in [m]-I}\overline{P}(B_{i},t)\Big]
\end{equation*}
\nd where $\overline{P}\big(\widehat{D}_{\underline{C},\underline{E}}^{I}(\sigma)\big)$ 
can be read off from \eqref{eqn:ced.sigma}. 
\end{thm}

Finally, Theorem \ref{thm:bbcgsplitting} gives now the  
Hilbert-Poincar\'e series for $Z(L;(U,V))$ and for $Z(L;(X,A)$, by applying
\eqref{eqn:hpszhat} for each $K = L_J, \; J \subseteq [m]$. 

\section{\hspace{0.8mm}Applications}\label{sec:applications}
\begin{exm}\label{exm:cp4additive}
Consider the composite
 \begin{equation}\label{eqn:cp2cp3}
 f\colon \mathbb{C}P^3 \rightarrow \mathbb{C}P^{3}/\mathbb{C}P^1 
 \xhookrightarrow{\iota} \mathbb{C}P^{8}/\mathbb{C}P^1
 \end{equation}
 where the map $\iota$ is the inclusion of the bottom two cells.
 \nd Denote the mapping cylinder of \eqref{eqn:cp2cp3} by $M_f$ and consider the pair
$(M_f, \mathbb{C}P^3)$ for which we shall determine
$\widetilde{H}^{*}\big(\widehat{Z}\big(K;(M_f, \mathbb{C}P^3)\big)\big)$, and hence, 
$H^{\ast}\big(Z(K;(M_f, \mathbb{C}P^3)\big)\big)$ for any
simplicial complex $K$ on vertices $[m]$. Here,
$$(U,V) = \Big(\mvee_{k=2}^{3}S^{2k} \vee  \mvee_{k=4}^{8}S^{2k},\; 
\;\mvee_{k=2}^{3}S^{2k} \vee S^{2}\Big).$$
so that  $B = \mvee_{k=2}^{3}S^{2k}$, $C= \mvee_{k=4}^{8}S^{2k}$ and $E = S^{2}$. 
Theorem \ref{thm:cartan2} gives now
$$\widetilde{H}^{*}\big(\widehat{Z}\big(K;(M_f, \mathbb{C}P^3)\big)\big) \cong 
\widetilde{H}^{*}\big(\widehat{Z}(K;(\underline{U},\underline{V}))\big).$$
Applying Theorem \ref{thm:main}, we get
\begin{align*}\widehat{Z}(K;(U,V)) &\xrightarrow{\simeq} 
\mvee_{I\leq [m]}\widehat{Z}\big(K_{I};\big(\mvee_{k=4}^{8}S^{2k}, \;S^{2}\big)\big)\wedge 
\widehat{Z}\big(K_{[m]- I};(\mvee_{k=2}^{3}S^{2k},\;\mvee_{k=2}^{3}S^{2k}\big)\big)\\
&= \mvee_{I\leq [m]}\widehat{Z}\big(K_{I};\big(\mvee_{k=4}^{8}S^{2k},\;S^{2}\big)\big)
\wedge (\mvee_{k=2}^{3}S^{2k}\big)^{{\wedge}|[m]-I|}
\end{align*}
\nd where the last term represents the $(|[m]-I)|$-fold smash product.
Finally, Corollary \ref{cor:wedge} determines  completely each term
$$\widehat{Z}\big(K_{I};\big(\mvee_{k=4}^{8}S^{2k}, \;S^{2}\big)\big)$$
by enumerating all  the links $|lk_{\sigma}(K_I)|$.
\end{exm}
Theorem \ref{thm:main} applies particularly well in cases where spaces have unstable attaching maps.
\begin{exm}
The homotopy equivalence $S^{1}\wedge Y \simeq \Sigma(Y)$ implies homotopy equivalences
\begin{equation}\label{eqn:he1}
\Sigma^{mq}\big(\widehat{Z}(K;(\underline{X},\underline{A}))\big)
\longrightarrow \widehat{Z}\big(K;\big(\underline{\Sigma^{q}(X)},\underline{\Sigma^{q}(A)}\big)\big)
\end{equation}
\nd where as usual,  $m$  is the number of vertices of $K$. Recall now that 
$SO(3) \cong \mathbb{R}\rm{P}^{3}$ 
and consider the pair
$$(X,A) = \big(SO(3), \mathbb{R}\rm{P}^{2}\big),$$
\nd for which there is a well known homotopy equivalence of pairs,  \cite[Section 1]{mukai}, 
\begin{equation}\label{eqn:he2}
\big(\Sigma^{2}\big(SO(3)\big), \;\Sigma^{2}\big(\mathbb{R}\rm{P}^{2}\big)\big) \longrightarrow 
\big(\Sigma^{2}\big(\mathbb{R}\rm{P}^{2}\big)\vee \Sigma^{2}(S^{3}),\;
\Sigma^{2}\big(\mathbb{R}\rm{P}^{2}\big)\big),
\end{equation}
\nd which makes the pair $\big(SO(3), \mathbb{R}\rm{P}^{2}\big)$ {\em stably wedge decomposable\/}.
\nd Next, combining \eqref{eqn:he1} and \eqref{eqn:he2}, we get a homotopy equivalence
$$\Sigma^{2m}\big(\widehat{Z}(K;(SO(3),\mathbb{R}\rm{P}^{2}))\big)
\longrightarrow 
\widehat{Z}\big(K;\big(\Sigma^{2}(\mathbb{R}\rm{P}^{2})\vee \Sigma^{2}(S^{3}),\;
\Sigma^{2}(\mathbb{R}\rm{P}^{2})\big).$$
\nd Finally, Theorem \ref{thm:cartan2} allows us to conclude that $\widehat{Z}(K;(SO(3),\mathbb{R}\rm{P}^{2}))\big)$,
and hence the polyhedral product $Z(K;(SO(3),\mathbb{R}\rm{P}^{2}))$,
is stably a wedge of smash products of $S^{3}$ and $\mathbb{R}\rm{P}^{2}$.
\end{exm}
Similar splitting results exist for
the polyhedral product whenever the spaces $X$ and $A$ split after finitely many suspensions. In particular, the fact that
$\Omega^{2}S^{3}$ splits stably into a wedge of Brown--Gitler spectra implies that the polyhedral product 
$Z\big(K; (\Omega^{2}S^{3}, \ast)\big)$ splits  stably  into a wedge of smash products of Brown--Gitler spectra.

\section{\hspace{0.8mm}Product structure}\label{sec:products}
The purpose of this section is to describe the product structure of 
$\widetilde{H}^{*}\big(Z(K;(\underline{X},\underline{A}))\big)$
in terms of the names of the additive generators given by Theorem \ref{thm:cartan2}, Corollary \ref{cor:wedge}  and
\eqref{eqn:hochster} below. We continue to work under the strong freeness conditions of Definition 
\ref{defn:strongfc}, which are satisfied, in particular, if we take coefficients in a field $k$.

\subsection{{Background}}\label{subsec:back}
We begin with a brief summary of the properties of {\em partial diagonal\/}  maps from \cite{bbcg3} and \cite{bbcg10}.
The main theorem of \cite{bbcg3} asserts that the product structure on the algebra 
${H}^\ast\big(Z(K;(\underline{X},\underline{A}))\big)$, where $K$ is a simplicial complex on $[m]$, is determined 
completely by partial diagonals defined for  $P,Q$ subsets of $[m]$, (\cite[Section 2]{bbcg3}),
\begin{equation}\label{eqn:partialdiags}
\widehat{\Delta}_{P \cup Q}^{P,Q}\;\colon \widehat{Z}\big(K_{P\cup Q};(\underline{X},\underline{A})_{P\cup Q} \big) 
\longrightarrow \widehat{Z}\big(K_{P};(\underline{X},\underline{A})_{P} \big) \wedge 
\widehat{Z}\big(K_{Q};(\underline{X},\underline{A})_{ Q} \big)
\end{equation}
\nd  inducing
$$\widehat{H}^\ast\big(\widehat{Z}\big(K_{P};(\underline{X},\underline{A})_{P} \big)\big) \otimes
\widehat{H}^\ast\big(\widehat{Z}\big(K_{Q};(\underline{X},\underline{A})_{ Q} \big) \longrightarrow
\widehat{H}^\ast\big(\widehat{Z}\big(K_{P\cup Q};(\underline{X},\underline{A})_{P\cup Q} \big)\big).$$
A sketch of the ideas from \cite[Section 6]{bbcg10} follows next. A family of pairs 
$(\underline{X},\underline{A})^{P,Q}_{P\cup Q}$ \; is defined by
 \begin{equation}\label{eqn:doublingpairs}
\big[(\underline{X},\underline{A})^{P,Q}_{P\cup Q}\big]_i = \begin{cases}
(X_i,A_i) \quad& {\rm if}\; i \in (P\cup Q) \msetm (P \cap Q)\\
(X_i\wedge X_i, A_i \wedge A_i) & {\rm if}\; i \in P\cap Q. \end{cases} 
\end{equation}

 \nd The map $\widehat{\Delta}_{P \cup Q}^{P,Q}$,  factors as
 \begin{equation}\label{eqn:composite1}
 \widehat{Z}\big(K_{P\cup Q};(\underline{X},\underline{A})_{P\cup Q} \big)  
 \xrightarrow{\widehat{\psi}^{P,Q}_{P\cup Q}} 
 \widehat{Z}\big(K_{P\cup Q};(\underline{X},\underline{A})^{P,Q}_{P\cup Q}\big)
 \end{equation}
 \nd followed by
 \begin{equation}\label{eqn:composite2}
 \widehat{Z}\big(K_{P\cup Q};(\underline{X},\underline{A})^{P,Q}_{P\cup Q}\big)    
 \xrightarrow{\widehat{S}}  \widehat{Z}\big(K_{P};(\underline{X},\underline{A})_{P} \big) \wedge 
\widehat{Z}\big(K_{Q};(\underline{X},\underline{A})_{ Q} \big)
 \end{equation}
\nd where $\widehat{S}$ is a shuffle map,  originating from a natural rearrangement of smash product factors at the 
diagram level, \cite[Section 7]{bbcg3},  and \;
$\widehat{\psi}^{P,Q}_{P\cup Q} \colon (\underline{X},\underline{A}) \longrightarrow 
(\underline{X},\underline{A})^{P,Q}_{P\cup Q}$\; is induced by the map of pairs
\begin{equation}\label{eqn:diag}
(X_i,A_i) \longmapsto \begin{cases} (X_i,A_i) 
&\text{if}\; i \in  (P\cup Q) \msetm \;(P \cap Q)\\
(X_i\wedge X_i,A_i\wedge A_i) &\text{if}\; i \in P\cap Q. \end{cases}
\end{equation}
The connection between the partial diagonals and the cup product in 
${H}^\ast\big(Z(K;(\underline{X},\underline{A}))\big)$ is given by the next commutative 
diagram of diagonals and projections, \cite[Section 1]{bbcg3}.
\begin{equation}\label{eqn:starequalscup}
\hspace{0.0cm}\begin{tikzcd}
Z(K;(\underline{X},\underline{A}))\arrow[r, "\Delta_{Z(K;(\underline{X},\underline{A}))}"]
\arrow[d, "\widehat{\Pi}_{P\cup Q}"]
&Z(K;(\underline{X},\underline{A})) \wedge Z(K;(\underline{X},\underline{A}))
\arrow[d, "\widehat{\Pi}_P \;\wedge\; \widehat{\Pi}_Q"]\\ 
\widehat{Z}\big(K_{P\cup Q};(\underline{X},\underline{A})_{P\cup Q}\big)
\arrow[r, "\widehat{\Delta}_{P \cup Q}^{P,Q}"]
&\widehat{Z}\big(K_{P};(\underline{X},\underline{A})_{P} \big) \wedge
\widehat{Z}\big(K_{Q};(\underline{X},\underline{A})_{Q} \big)
\end{tikzcd}
\end{equation}
\skp{0.1}
\nd The projection maps $\widehat{\Pi}_I$ are induced from the composition of the two projections:
\begin{equation}\label{eqn:twoprojections}
Y^{[m]} \longrightarrow Y^I \longrightarrow \widehat{Y}^I.
\end{equation}
This leads to the definition of the $\ast$-product.
\skp{0.1}
\begin{defin}\label{defn: starproduct}
\nd Given cohomology classes 
$$u \in \widetilde{H}^{p}\big(\widehat{Z}\big(K_{P};(\underline{X},\underline{A})_{P} \big)\big)\;\;\text{and}\;\;
v \in \widetilde{H}^{q}\big(\widehat{Z}\big(K_{Q};(\underline{X},\underline{A})_{Q} \big)\big),$$ 
the star product $\ast$-product is defined by
$$u\ast v = (\widehat{\Delta}_{P \cup Q}^{P,Q})^\ast(u\otimes v) \in 
\widetilde{H}^{p+q}\big(\widehat{Z}\big(K_{P\cup Q};(\underline{X},\underline{A})_{P\cup Q}\big)\big).$$
\end{defin}

\nd The commutativity of diagram \eqref{eqn:starequalscup} implies that
\begin{equation}\label{eqn:startocup}
(\widehat{\Pi}_{P\cup Q})^{\ast}(u\ast v) = (\widehat{\Pi}_{P})^{\ast}(u)
\smallsmile (\widehat{\Pi}_{Q})^{\ast}(v)
\end{equation}
where $\smallsmile$ denotes the cup product in $\widetilde{H}^\ast\big(Z(K;(\underline{X},\underline{A}))\big)$.
Finally, the $\ast$-product endows  
{\small $$ \moplus_{I\subseteq [m]}
\widetilde{H}^{\ast}\big(\widehat{Z}\big(K_{I};(\underline{X},\underline{A})_{I}\big)\big)$$}
with a ring structure, and there is an isomorphism of rings
\begin{equation}\label{eqn:hochster}
\eta\; \colon  \moplus_{I\subseteq [m]}
\widetilde{H}^{\ast}\big(\widehat{Z}\big(K_{I};(\underline{X},\underline{A})_{I}\big)\big) \longrightarrow 
\widetilde{H}^\ast\big(Z(K;(\underline{X},\underline{A}))\big) 
\end{equation}
induced from the additive isomorphism of Theorem \ref{thm:bbcgsplitting}, as described in 
\cite[Section 1]{bbcg3}.
\skp{0.3}

\subsection{{Partial diagonals, the wedge lemma and the product of links}}\label{subsec:linkprods}
We discuss now the partial diagonal maps given by \eqref{eqn:partialdiags} in the context of 
Theorem \ref{thm:wlemma}, which asserts that if every $E_i \hookrightarrow C_i$ 
is null homotopic then for any $I \subseteq [m]$,
\begin{equation}\label{eqn:wedgedecompce}
\widehat{Z}\big(K_{I};(\underline{C},\underline{E})_I\big) \simeq 
\mvee_{\sigma \in K_{I}}|\Delta((K_{I})_{<\sigma})|\ast
\widehat{D}_{\underline{C},\underline{E}}^{I}(\sigma) \simeq
\mvee_{\sigma \in K_{I}} |lk_{\sigma}(K_{I})|\ast
\widehat{D}_{\underline{C},\underline{E}}^{I}(\sigma)
\end{equation}
\nd where $\widehat{D}_{\underline{C},\underline{E}}^{I}(\sigma)$ is as in \eqref{eqn:ced.sigma}.
The next lemma checks that the partial diagonal maps behave as expected on the summands given by the wedge
lemma.

\begin{lem}\label{lem:diagsbehave}
The proof of the wedge lemma, \cite[Section 4]{zz}, exhibits vertical embeddings of 
homotopy colimts so that the following diagram commutes for $I = J\cup L$, 
$\tau = \sigma \cap J$ and $\omega = \sigma \cap L$,
{\small \begin{equation}\label{eqn:cediag}
\begin{tikzcd}
\widehat{Z}\big(K_{I};(\underline{C},\underline{E})_{I}\big)
\arrow[r, "\widehat{\Delta}_{I}^{J,L}"]
&\widehat{Z}\big(K_{J};(\underline{C},\underline{E})_{J} \big) \wedge
\widehat{Z}\big(K_{L};(\underline{C},\underline{E})_{L} \big)\\ 
\mid\hspace{-1mm}{lk}_{\sigma}(K_{I})\hspace{-1.2mm}\mid\ast
\widehat{D}_{\underline{C},\underline{E}}^{I}(\sigma)
\arrow[r, "\widehat{\xi}_{I}^{J,L}"]\arrow[u, "\overline{\gamma(\sigma)}"]
&\mid\hspace{-1mm}{lk}_{\tau}(K_{J})\hspace{-1.2mm}\mid\ast
\widehat{D}_{\underline{C},\underline{E}}^{J}(\tau) \;\wedge\;
\mid\hspace{-1mm}{lk}_{\omega}(K_{L})\hspace{-1.2mm}\mid\ast
\widehat{D}_{\underline{C},\underline{E}}^{L}(\omega)
\arrow[u, "\overline{\gamma(\tau)}\;\wedge\;\overline{\gamma(\omega)}"']
\end{tikzcd}
\end{equation}}
\end{lem}
\begin{proof}
The map $\widehat{\xi}_{I}^{J,L}$ is the restriction of $\widehat{\Delta}_{I}^{J,L}$.  
\end{proof}
\nd In particular, \eqref{eqn:cediag} identifies the target under the partial diagonal $\widehat{\Delta}_{I}^{J,L}$, for every
wedge summand of $\widehat{Z}\big(K_{I};(\underline{C},\underline{E})_{I}\big)$.
In order to better understand the map $\widehat{\xi}_{I}^{J,L}$ from \eqref{eqn:cediag} and its effect on the 
cohomology of the links, we need to consider a polyhedral product which contains all the links.
\skp{0.1}
 Consider next the space $\mz = D^1 \vee S^0$ and let 
\begin{equation}\label{eqn:s0}
\iota\colon S^0 \hookrightarrow D^1 \hookrightarrow \mz
\end{equation}
be the basepoint preserving inclusion which takes $S^0$ to the endpoints of $D^1$; this gives a pair $(\mz, S^0)$. 
Next, let $\mathcal{K}$ be a simplicial complex on $[m]$ and apply Theorem \ref{thm:wlemma} to the polyhedral 
smash product $\widehat{Z}(\mathcal{K};(\mz, S^0))$ to get 
\begin{equation}\label{eqn:sumoflinks}
\widehat{Z}(\mathcal{K};(\mz, S^0)) \;\simeq\; \mvee_{\sigma \in \mathcal{K}} |lk_{\sigma}(\mathcal{K})|\ast
\widehat{D}_{\umz,\; \underline{S}^0}^{[m]}(\sigma) \;\simeq\; 
\mvee_{\sigma \in \mathcal{K}}\Sigma{|lk_{\sigma}(\mathcal{K})|},
\end{equation} 
since  $\widehat{D}_{\umz,\; \underline{S}^0}^{[m]}(\sigma) \simeq S^0$. 
For each $\sigma \in \mathcal{K}$, the inclusion of  
each wedge summand given by Theorem \ref{thm:wlemma} can be described as follows
{\fontsize{11}{8}\selectfont
\begin{equation}\label{eqn:szero}
\mid\hspace{-1mm}{lk}_{\sigma}(\mathcal{K})\hspace{-1.2mm}\mid\ast
\widehat{D}_{\underline{C},\underline{E}}^{[m]}(\sigma) =
\Sigma{\mid\hspace{-1mm}{lk}_{\sigma}(\mathcal{K})\hspace{-1.2mm}\mid}\wedge
\widehat{D}_{\underline{C},\underline{E}}^{[m]}(\sigma) \xhookrightarrow{\iota}
\widehat{Z}\big(\mathcal{K};(\mz,S^0)\big) \wedge \widehat{D}_{\underline{C},\underline{E}}^{[m]}(\sigma).
\end{equation}}

Combining  diagram \eqref{eqn:starequalscup} with \eqref{eqn:sumoflinks}, we get the
 commutative diagram below.
{\fontsize{11}{8}\selectfont \begin{equation}\label{eqn:szerolinks}
\hspace{-0.1in}\begin{tikzcd}[column sep=7.5ex, row sep=3ex]
Z\big(\mathcal{K};(\mz,S^0)\big) \arrow[r, "\Delta_{Z(\mathcal{K};(\mz,S^0))}"]
\arrow[d, "\;\widehat{\Pi}_{I}"]
&Z\big(\mathcal{K};(\mz,S^0)\big) \wedge Z\big(\mathcal{K};(\mz,S^0)\big)
\arrow[d, "\;\widehat{\Pi}_J \;\wedge\; \widehat{\Pi}_L"]\\ 
\widehat{Z}\big(\mathcal{K}_I;(\mz,S^0)_{I}\big) 
\arrow[r, "\widehat{\Delta}_{I}^{J,L}"] \arrow[d, "\;\pi_{\sigma}"]
&\widehat{Z}\big(\mathcal{K}_J;(\mz,S^0)_{J}\big) \wedge
\widehat{Z}\big(\mathcal{K}_L;(\mz,S^0)_{L}\big)
\arrow[d, "\;\pi_\tau\;\wedge\;\pi_\omega"]\\
\Sigma{\mid\hspace{-1mm}{lk}_{\sigma}(\mathcal{K}_{I})\hspace{-1.2mm}\mid}
\arrow[r, "\widehat{\xi}_{I}^{J,L}"]
&\Sigma{\mid\hspace{-1mm}{lk}_{\tau}(\mathcal{K}_{J})\hspace{-1.2mm}\mid} \;\wedge\;
\Sigma{\mid\hspace{-1mm}{lk}_{\omega}(\mathcal{K}_{L})\hspace{-1.2mm}\mid}
\end{tikzcd}
\end{equation}}
\nd where the vertical maps at the bottom are projections onto wedge summands. This motivates the next definition.

\begin{defin}
Let $\alpha \in \widetilde{H}^{\ast}(|{lk}_{\tau}(\mathcal{K}_{J})|)$ and
$\beta \in \widetilde{H}^{\ast}(|{lk}_{\omega}(\mathcal{K}_{L})|)$,
we write
$$\alpha \ast  \beta \;\becomes 
(\widehat{\xi}_{I}^{J,L})^\ast(\alpha_{\tau,J} \otimes \alpha_{\omega,L} )$$
\end{defin}
We need more information about $\widetilde{H}^{\ast}\big(Z(K;(\mz,S^0))\big)$ before undertaking an example.
\begin{lem}
We have the following isomorphism of cohomology groups for the simplicial complex $K$ consisting of the empty 
simplex and $m$ discrete points
$$\widetilde{H}^{\ast}\big(Z(K;(\mz,S^0))\big) \xrightarrow{\cong}
\moplus_{I\subset [m]}\Big(\widetilde{H}^{\ast}\big(\mvee_{|I|-1}S^1\big)
\nonumber\moplus_{\{i\} \in K_I}\widetilde{H}^{\ast}\big(\Sigma{\varnothing}\big)\Big)\\ \nonumber$$
\end{lem}
\begin{proof}
Here  $K = \big\{\varnothing, \{1\}, \{2\}, \ldots, \{m\}\big\}$.
Theorems  \ref{thm:bbcgsplitting} and \ref{thm:wlemma} give
\begin{align}\label{eqn:homszero}
\nonumber\widetilde{H}^{\ast}\big(Z(K;(\mz,S^0))\big) 
\nonumber&\cong \moplus_{I\subset [m]} \widetilde{H}^{\ast}\big(\widetilde{Z}(K_I;(\mz, S^0)_I)\big)\\
\nonumber&\cong \moplus_{I\subset [m]}\Big(\moplus_{\sigma \in K_I}\widetilde{H}^{\ast}\big(|lk_{\sigma}(K_{I})|\ast
\nonumber\widetilde{D}_{\mz,S^0}^{I}(\sigma)\big) \Big)\\
&\cong \moplus_{I\subset [m]}\Big(\moplus_{\sigma \in K_I}\widetilde{H}^{\ast}\big(\Sigma|lk_{\sigma}(K_{I})|\big)\Big)\\
\nonumber&\cong \moplus_{I\subset [m]}\Big(\widetilde{H}^{\ast}\big(\Sigma|lk_{\varnothing}(K_{I})|\big)
\nonumber\moplus_{\{i\} \in K_I}\widetilde{H}^{\ast}\big(\Sigma|lk_{\{i\}}(K_{I})|\big)\Big)\\ 
\nonumber&\cong \moplus_{I\subset [m]}\Big(\widetilde{H}^{\ast}\big(\mvee_{|I|-1}S^1\big)
\nonumber\moplus_{\{i\} \in K_I}\widetilde{H}^{\ast}\big(\Sigma{\varnothing}\big)\Big)\\ \nonumber
\end{align}
\skp{-0.6}
\nd  where $\widetilde{D}_{\mz,S^0}^{I}(\sigma) \simeq S^0$ is as in \eqref{eqn:ced.sigma}. \end{proof}

\begin{exm}\label{exm:discretepts}
We shall now analyze diagram \eqref{eqn:szerolinks} for the case that
$\mathcal{K} = K$ is the simplicial complex consisting of the empty simplex and $m$ discrete points
$$K = \big\{\varnothing, \{1\}, \{2\}, \ldots, \{m\}\big\}.$$
Set $I \subset [m]$,  $J$ and $L$ subsets of $I$ satisfying $J \cup L = I$. Let $\sigma \in K_I$
be a simplex, we wish to compute
$$\widetilde{H}^{\ast}\big(\Sigma\mid\hspace{-1mm}{lk}_{\sigma}(K_I)\hspace{-1.2mm}\mid\big)
\xleftarrow{\widehat{\xi}_{I}^{J,L})^{\ast}}
\widetilde{H}^{\ast}\big(\Sigma\mid\hspace{-1mm}{lk}_{\tau}(K_J)\hspace{-1.2mm}\mid\big)
\;\motimes\;
\widetilde{H}^{\ast}\big(\Sigma\mid\hspace{-1mm}{lk}_{\omega}(K_L)\hspace{-1.2mm}\mid\big)$$
for $\tau = \sigma\cap J$ and  $\omega = \sigma \cap L$. There are cases to consider:
\begin{enumerate}\itemsep3mm
\item $\sigma = \{i\} \in K_I$
\begin{enumerate}[(a)]\itemsep2mm
\item $\tau = \sigma\cap J = \{i\}$, $\omega = \sigma \cap L = \{i\}$, so $lk_{\tau}(K_J)  = \varnothing$
and $lk_{\omega}(K_L) =\varnothing$.
\item $\tau = \sigma\cap J =\varnothing$, 
$ \omega =\sigma \cap L = \{i\}$, so $lk_{\tau}(K_J) = K_J$
 and $ lk_{\omega}(K_L)= \varnothing$. 
\end{enumerate}
\skp{0.15}
\item $\sigma = \varnothing \in K_I$
\begin{enumerate}\itemsep2mm
\item[Here,] $\tau = \sigma\cap J = \varnothing$, $\omega = \sigma \cap L = \varnothing$, so 
$lk_{\tau}(K_J) = K_J$ and $lk_{\omega}(K_L) = K_L$.
\end{enumerate}
\end{enumerate}
\nd For each case, we wish to compute
\begin{equation}\label{eqn:gens}
\widetilde{H}^{\ast}\big(\Sigma\mid\hspace{-1mm}{lk}_{\sigma}(K_I)\hspace{-1.2mm}\mid\big)
\xleftarrow{\widehat{\xi}_{I}^{J,L})^{\ast}}
\widetilde{H}^{\ast}\big(\Sigma\mid\hspace{-1mm}{lk}_{\tau}(K_J)\hspace{-1.2mm}\mid\big)
\;\motimes\; \widetilde{H}^{\ast}\big(\Sigma\mid\hspace{-1mm}{lk}_{\omega}(K_L)\hspace{-1.2mm}\mid\big).
\end{equation}
\nd {\fontsize{11}{8}{\bf Case[1(a)]:}} In this case the bottom row of \eqref{eqn:szerolinks}, 
which we wish to determine in
cohomology, is
\begin{equation}\label{eqn:case1a}
\begin{tikzcd}[column sep=5ex, row sep=-0.5ex]
\widetilde{H}^{0}(\Sigma{\varnothing})
&\arrow[l, "(\widehat{\xi}_{I}^{J,L})^{\ast}"']
\widetilde{H}^{0}(\Sigma{\varnothing})
\;\motimes\;
\widetilde{H}^{0}(\Sigma{\varnothing})\\
\cong k 
&\hspace{-0.145in}
\cong k \otimes k.
\end{tikzcd}
\end{equation}
We denote the unit generators in \eqref{eqn:case1a} by $\alpha_{i,I}$, $\alpha_{i,J}$ and
$\alpha_{i,L}$ respectively. Applying cohomology to diagram \eqref{eqn:szerolinks}, we see that
$$(\widehat{\Pi}_I^{\ast} \circ   \pi_{\sigma}^{\ast})(\alpha_{i,J} \ast \alpha_{i,L}) = 
(\widehat{\Pi}_J^{\ast} \circ \pi_{\tau}^{\ast})(\alpha_{i,J}) \smile 
(\widehat{\Pi}_L^{\ast} \circ \pi_{\omega}^{\ast})(\alpha_{i,L}).$$
Corresponding to these classes, the top row of the diagram restricts to the cup product isomorphism
\begin{equation}
\begin{tikzcd}[column sep=13ex, row sep=-0.5ex]
\widetilde{H}^{0}(S^0)
&\arrow[l, "\Delta^{\ast}_{Z(K;(\mz,S^0))}"'] \arrow[l, "\cong"]
\widetilde{H}^{0}(S^0)
\;\motimes\;
\widetilde{H}^{0}(S^0)\\
\cong k 
&\hspace{-0.18in}
\cong k \otimes k.
\end{tikzcd}
\end{equation}
Now, since the map $(\iota_{\tau})^{\ast}\otimes (\iota_{\omega})^{\ast}$ is an isomorphism in this example,
we conclude in this case that
\begin{equation}\label{eqn:firstalpha}
\alpha_{i,J} \ast \alpha_{i,L}\;=\; \alpha_{i,I}.
\end{equation}

\nd {\fontsize{11}{8}{\bf Case[1(b)]:}} In this case the bottom row of \eqref{eqn:szerolinks} in cohomology, is
\begin{equation}\label{eqn:case1b}
\begin{tikzcd}[column sep=5ex, row sep=-0.5ex]
\widetilde{H}^{0}(\Sigma{\varnothing})
&\arrow[l, "(\widehat{\xi}_{I}^{J,L})^{\ast}"']
\widetilde{H}^{\ast}(\Sigma{K_J})
\;\motimes\;
\widetilde{H}^{0}(\Sigma{\varnothing})\\
\cong k 
&\hspace{-0.05in}
\cong H^{\ast}\big(\ds{\mvee_{|J|-1}}S^1\big) \otimes k.
\end{tikzcd}
\end{equation}
This time, we denote the unit generators in \eqref{eqn:case1b} by $\alpha_{i,I}$, $\beta_{i,J}$ and
$\alpha_{i,L}$ respectively. Applying cohomology to diagram \eqref{eqn:szerolinks}, we see that
$$(\widehat{\Pi}_I^{\ast} \circ   \pi_{\sigma}^{\ast})(\beta_{i,J} \ast \alpha_{i,L}) = 
(\widehat{\Pi}_J^{\ast} \circ \pi_{\tau}^{\ast})(\beta_{i,J}) \smile 
(\widehat{\Pi}_L^{\ast} \circ \pi_{\omega}^{\ast})(\alpha_{i,L})$$
\nd which is zero for dimensional reasons and we conclude
\begin{equation}\label{eqn:2ndalpha}
\beta_{i,J} \ast \alpha_{i,L}\;=\; 0.
\end{equation}

\nd {\fontsize{11}{8}{\bf Case[2]:}} This time the bottom row of \eqref{eqn:szerolinks}, in cohomology, is
\begin{equation}\label{eqn:case2}
\begin{tikzcd}[column sep=5ex, row sep=-0.5ex]
\widetilde{H}^{\ast}(K_I)
&\arrow[l, "(\widehat{\xi}_{I}^{J,L})^{\ast}"']
\widetilde{H}^{\ast}(K_J)
\;\motimes\;
\widetilde{H}^{ast}(K_L)\\
\cong H^{\ast}\big(\ds{\mvee_{|I|-1}}S^1\big) 
&\hspace{-0.01in}
\cong H^{\ast}\big(\ds{\mvee_{|J|-1}}S^1\big) \otimes H^{\ast}\big(\ds{\mvee_{|L|-1}}S^1\big).
\end{tikzcd}
\end{equation}
The unit generators in \eqref{eqn:case2} are denoted this time by $\beta_{i,I}$, $\beta_{i,J}$ and
$\beta_{i,L}$ respectively. Applying cohomology to diagram \eqref{eqn:szerolinks}, we see that
$$(\widehat{\Pi}_I^{\ast} \circ   \pi_{\sigma}^{\ast})(\beta_{i,J} \ast \beta_{i,L}) = 
(\widehat{\Pi}_J^{\ast} \circ \pi_{\tau}^{\ast})(\beta_{i,J}) \smile 
(\widehat{\Pi}_L^{\ast} \circ \pi_{\omega}^{\ast})(\alpha_{i,L})$$
\nd which is zero for dimensional reasons and so we conclude
\begin{equation}\label{eqn:2ndbeta}
\beta_{i,J} \ast \beta_{i,L}\;=\; 0.
\end{equation}
This ends the discussion of Example \ref{exm:discretepts}.
\end{exm}
An alternative addition to the toolkit for computing the $\ast$-products of links is outlined in the following remark.
\begin{remark}
Let $\mathcal{K}$ be a simplicial complex on $[m]$, $I \subset [m]$ and $\sigma \in \mathcal{K}_I$. For a pair of subsets  $J$ and $L$ 
of $[m]$ satisfying $I = J\cup L$, set  $\tau = \sigma \cap J$ and $\omega = \sigma \cap L$. 
Then $(\mathcal{K}_I)_J = \mathcal{K}_J, (\mathcal{K}_I)_L = \mathcal{K}_L$ and $\big({lk}_{\sigma}(\mathcal{K})\big)_I = {lk}_{\sigma}(\mathcal{K}_I)$. There  are also natural 
inclusions
\begin{equation}\label{eqn:linkfull}
{\begin{tikzcd}[column sep=3ex, row sep=-0.5ex]
\big({lk}_{\sigma}(\mathcal{K}_I)\big)_J \arrow[r, hook, "\iota_{\tau}"] &lk_{\tau}(\mathcal{K}_J) 
\hspace{0.13in}{\rm and}\hspace{-0.2in}
&\big({lk}_{\sigma}(\mathcal{K}_I)\big)_L \arrow[r, hook, "\iota_{\omega}"] &lk_{\omega}(\mathcal{K}_L).
\end{tikzcd}}
\end{equation}\nd
In cases  when these inclusions are equivalences of simplicial complexes, information about link products may be gleaned
from the commutative diagram

{\fontsize{11}{8}\selectfont \begin{equation*}
\begin{tikzcd}[column sep=10ex, row sep=2.5ex]
Z({lk}_{\sigma}(K);(D^1,S^0)) \arrow[r, "\Delta_{Z({lk}_{\sigma}(K);(D^1,S^0)}"] 
\arrow[d, "\;\widehat{\Pi}_I"]
&Z({lk}_{\sigma}(K);(D^1,S^0)) \wedge
Z({lk}_{\sigma}(K);(D^1,S^0)) \arrow[d, "\;\widehat{\Pi}_J \wedge \widehat{\Pi}_L"] \\
\widehat{Z}\big(({lk}_{\sigma}(K))_I;(D^1,S^0)\big)
\arrow[r, "\widehat{\Delta}_{I}^{J,L}"] \arrow[d, "\cong"]
&\widehat{Z}\big(({lk}_{\sigma}(K))_J;(D^1, S^0)\big) \wedge 
\widehat{Z}\big(({lk}_{\sigma}(K))_L;(D^1,S^0)\big) \arrow[d, "\cong"']\\
\Sigma\mid\hspace{-1mm}\big({lk}_{\sigma}(K)\big)_I\hspace{-1.2mm}\mid
\arrow[r, "\widehat{\Delta}_{I}^{J,L}"]
&\Sigma\mid\hspace{-1mm}\big({lk}_{\sigma}(K)\big)_J\hspace{-1.2mm}\mid
\;\wedge\;
\Sigma\mid\hspace{-1mm}\big({lk}_{\sigma}(K)\big)_L\hspace{-1.2mm}\mid.
\end{tikzcd}
\end{equation*}}

\nd where the lower vertical isomorphisms are given by Theorem \ref{thm:wlemma}. In Example \ref{exm:discretepts},
the inclusions \eqref{eqn:linkfull} are equivalences in cases [1(a)] and [2] but not in case [1(b)]. In that case however
a dimension argument can be used to get \eqref{eqn:2ndalpha}.
\end{remark}

\subsection{{Product structure for wedge decomposable pairs}\label{sec:wdpairsprod}} 
We consider now pairs of the form $(\underline{U},\underline{V}) = (\underline{B\vee C},\underline{B\vee E})$ where
the inclusion $E_i \hookrightarrow C_i$ is null homotopic for all $i \in \{1,2,\ldots,m\}$, as in Definition \ref{def:wedgedecomp}. 
We begin with the special case  $B = \ast$ (a point), and compute the $\ast$-product.
In turn, this suffices to determine the ring structure of
$$\widetilde{H}^{\ast}\big(Z(K;(\underline{U},\underline{V}))\big)\;=\; 
\widetilde{H}^{\ast}\big(Z(K;(\underline{C},\underline{E}))\big)$$
by the ring isomorphism \eqref{eqn:hochster}. The $\ast$-product is given by  diagram \eqref{eqn:cediag}
above, based on the decomposition of $\widetilde{H}^{\ast}\big(Z(K;(\underline{C},\underline{E}))\big)$ given by
Corollary \ref{cor:wedge}.

As usual we set $I \subset [m]$,  $J$ and $L$ subsets of $I$ satisfying $J \cup L = I$. Let $\sigma \in K_I$
be a simplex. Our goal then is to compute
\begin{multline}\label{eqn:wedgeclasses}
\widetilde{H}^{\ast}\big(\Sigma{|{lk}_{\tau}(K_J)|}\big)
\otimes H^{\ast}\big(\widehat{D}_{\underline{C},\underline{E}}^{J}(\tau)\big) 
\motimes 
\widetilde{H}^{\ast}\big(\Sigma{|{lk}_{\omega}(K_L)|}\big)
\otimes H^{\ast}\big(\widehat{D}_{\underline{C},\underline{E}}^{L}(\omega)\big) \\
\xrightarrow{\widehat{\xi}_{I}^{J,L})^{\ast}}\;\;\widetilde{H}^{\ast}\big(\Sigma{|{lk}_{\sigma}(K_I)|}\big)
\otimes H^{\ast}\big(\widehat{D}_{\underline{C},\underline{E}}^{I}(\sigma)\big) 
\end{multline}
for $\tau = \sigma\cap J$ and  $\omega = \sigma \cap L$. Consider a class
$$u\otimes v \in \widetilde{H}^{\ast}\big(\widehat{Z}\big(K_{J};(\underline{C},\underline{E})_{J} \big)\big) \otimes
\widetilde{H}^{\ast}\big(\widehat{Z}\big(K_{L};(\underline{C},\underline{E})_{L} \big)\big).$$
in the left hand side of \eqref{eqn:wedgeclasses}, it has the form
\begin{equation}\label{eqn:twoclasses}
u\otimes v  \;=\; \big(\alpha \otimes  \motimes_{i\in \tau}c_i \otimes 
\motimes_{j\in J\msetm \tau}e_j\big) \otimes
\big(\beta \otimes  \motimes_{k\in \omega}c_k \otimes 
\motimes_{l\in L\msetm \omega}e_l\big)
\end{equation}
\nd  where $\alpha$, $\beta$ are classes in $\big(\widetilde{H}^{\ast}(\Sigma{|lk_{\tau}K_J|})$ and 
$\big(\widetilde{H}^{\ast}(\Sigma{|lk_{\omega}K_L|})$ respectively.
(Here the cohomology grading of all the classes has been suppressed.) Aside from classes in the cohomology of 
the links, the only cohomology classes which are multiplied are those specified by the
diagonals  \eqref{eqn:diag} which arise on the  intersection $J\cap L$. That is
\begin{multline*}(\widehat{\Delta}_{J \cup L}^{J,L})^{\ast}(u\otimes v) = (\alpha\ast \beta) \otimes \hspace{-1mm}
\motimes_{i = k \in \tau\cap \omega}\hspace{-1mm}c_{i}c_{k}\;\otimes \hspace{-7mm}
\motimes_{i \neq k \in (\tau\cup\/\omega)\msetm (\tau\cap\/\omega)}\hspace{-7mm}c_{i} \otimes c_{k}
\;\otimes \hspace{-3mm}
\motimes_{j = l \in \tau'\cap \omega'}\hspace{-2mm}e_{j}e_{l}\;\otimes \hspace{-5mm}
\motimes_{j \neq l \in (\tau'\cup \omega')\msetm (\tau'\cap\omega')}\hspace{-6mm}e_{j} \otimes e_{l}\\[3mm]
\in \widetilde{H}^{\ast}(\Sigma{|lk_{\sigma}K_{I}|)}\;
\otimes\; \widetilde{H}^{\ast}(\widehat{D}_{\underline{C},\underline{E}}^{I}(\sigma)).\phantom{mmmmmmmmmm}
\end{multline*}
\nd where $\tau'$ and $\omega'$ denote $J\msetm\tau$ and $L\msetm \omega$ respectively.
We arrive now at the next lemma.
\begin{lem}\label{lem:ce}
The product structure  of the ring $\widetilde{H}^{\ast}\big(Z(K;(\underline{C},\underline{E}))\big)$ is determined 
by the $\ast$-product.
On two classes given as in  \eqref{eqn:twoclasses}, it is evaluated by taking the $\ast$-product of the link classes, 
(cf. subsection \ref{subsec:linkprods}), and
the ordinary cohomology product, coordinate-wise on the other cohomology factors of \eqref{eqn:wedgeclasses}
which correspond.\hspace{1.9truein}$\square$
\end{lem}
\begin{rem}
\nd Though this formula is concise, the computation of the link product $\alpha\ast \beta$ can be an obstacle. It is 
described in more detail in \cite[Section 7]{bbcg10} from the point of view of L.~Cai's work \cite{cai1}. 
From the more practical perspective of subsection \ref{subsec:linkprods}, diagram \eqref{eqn:szerolinks}
relates the $\ast$-product of these links to an actual cup product in $\widetilde{H}^{\ast}\big(Z(K;(\mz,S^0))\big)$ 
via \eqref{eqn:starequalscup}, which can often be checked explicitly. We shall take this approach below.
\end{rem}

\skp{0.1}
We consider next a full wedge decomposable pair $(\underline{U},\underline{V}) = (\underline{B\vee C},\underline{B\vee E})$
beginning with the part of the partial diagonals given by the shuffle 
maps from \eqref{eqn:composite2}.
\begin{equation}\label{eqn:shufflemapwd}
 \widehat{Z}\big(K_{P\cup Q};(\underline{U},\underline{V})^{P,Q}_{P\cup Q}\big)    
 \xrightarrow{\widehat{S}}  \widehat{Z}\big(K_{P};(\underline{U},\underline{V})_{P} \big) \wedge 
\widehat{Z}\big(K_{Q};(\underline{U},\underline{V})_{ Q} \big)
\end{equation}
For the pairs $(\underline{U},\underline{V})$, \eqref{eqn:doublingpairs} becomes
\begin{equation*}
\big[(\underline{U},\underline{V})^{P,Q}_{P\cup Q}\big]_i = \begin{cases}
(U_i,V_i) = (B_i\vee C_i,B_i\vee E_i) \quad& {\rm if}\; i \in P\cup Q \msetm P \cap Q\\
\big(U_i\wedge U_i, V_i \wedge V_i\big) 
= ({\widehat{B}_i\vee \widehat{C_i}}, {\widehat{B}_i\vee \widehat{E}_i})
&{\rm if}\; i \in P\cap Q. \end{cases} 
\end{equation*}
where here
\begin{align}\label{eqn:newbce}
\widehat{B}_i&= B_i\wedge B_i\nonumber\\
\widehat{C}_i&= (C_i\wedge C_i) \vee (C_i\wedge B_i) \vee  (B_i \wedge C_i)\\
\widehat{E}_i&= (E_i\wedge E_i) \vee (E_i\wedge B_i) \vee  (B_i \wedge E_i).\nonumber
\end{align}
In particular, the pairs in $(\underline{U},\underline{V})^{P,Q}_{P\cup Q}$ are all wedge
decomposable. This motivates the notational convention following.
\begin{defin}\label{defn:tildepairs}
For any subsets $S,T$ subsets of [$m]$, we set the notation
$$(\underline{\widetilde{B}\vee \widetilde{C}},\underline{\widetilde{B}\vee \widetilde{E}})_{S\cup T}  
\;=\; (\underline{U},\underline{V})^{S,T}_{S\cup T}$$   
\nd where
\begin{equation*}
(\widetilde{B}_i\vee \widetilde{C}_i,\widetilde{B}_i\vee \widetilde{E}_i) 
\;\becomes\begin{cases}
(B_i\vee C_i,B_i\vee E_i) \quad& {\rm if}\; i \in (S\cup T) \msetm (S\cap T)\\
(\widehat{B}_i\vee \widehat{C}_i, \widehat{B}_i\vee \widehat{E}_i)
&{\rm if}\; i \in S\cap T. \end{cases} 
\end{equation*}

In that which follows in this section, we adopt additional notation following: 
\samepage{\begin{enumerate}
\item  $P$ and $Q$ are subsets of $[m]$
\item $I\subset P\cup Q$
\item  $I = J\cup L$ with $J\subset P$ and $L \subset Q,$
\end{enumerate}}
\nd Notice that in this notation,
\begin{equation}\label{eqn:pqandi}
(P\cup Q) \msetm I \;=\; \big((P\msetm J)\cup (Q\msetm L)\big) \msetm
\big(J\cap (Q\msetm L)\big) \cup \big(L\cap (P\msetm J)\big).
\end{equation}
Generally, the meaning of the notation should become clear from the context. Figure\;$1$ is a Venn
diagram illustrating these sets, among other things.
\end{defin}

Next, we use
\begin{equation}\label{eqn:uvshufflewd}
\widehat{Z}\big(K_{P\cup Q};(\underline{U},\underline{V})^{P,Q}_{P\cup Q}\big)  \;=\;
\widehat{Z}\big(K_{P\cup Q}; (\widetilde{B}\vee \widetilde{C},\widetilde{B}\vee \widetilde{E})_{P\cup Q} \big),
\end{equation}
and apply Theorem \ref{thm:main}  to the right hand side, ($P\cup Q$ here plays the role of $[m]$ in the theorem),
to exhibit the shuffle map \eqref{eqn:shufflemapwd} on a wedge summand of \eqref{eqn:uvshufflewd} as 
{\small\begin{multline}\label{eqn:shuffleforwd}
\widehat{Z}\big(K_{I};(\underline{\widetilde{C}},\underline{\widetilde{E}})_{I}\big)
\wedge \widehat{Z}\big(K_{P\cup Q- I};(\underline{\widetilde{B}},\underline{\widetilde{B}})_{P\cup Q-I}\big)\\
\hspace{-3cm}\xrightarrow{\phantom{m}\widehat{S}\phantom{m}}
\widehat{Z}\big(K_{J};(\underline{C},\underline{E})_J\big)
\wedge\widehat{Z}\big(K_{P- J};(\underline{B},\underline{B})_{P-J}\big)
\wedge \;\widehat{Z}\big(K_{L};(\underline{C},\underline{E})_L \big)
\wedge \widehat{Z}\big(K_{Q- L};(\underline{B},\underline{B})_{Q-L}\big)
\end{multline}}
\hspace{-3mm} where $J\cup L = I$. In each of the the partial diagonal maps, the shuffle map is 
preceded by the map $\widehat{\psi}^{J,L}_{J\cup L}$ of \eqref{eqn:composite1} and \eqref{eqn:diag}
\begin{equation}\label{eqn:psiuv}
 \widehat{Z}\big(K_{P\cup Q};(\underline{U},\underline{V})_{P\cup Q} \big)  
 \xrightarrow{\widehat{\psi}^{P,Q}_{P\cup Q}} 
 \widehat{Z}\big(K_{P\cup Q};(\underline{U},\underline{V})^{P,Q}_{P\cup Q}\big).
 \end{equation}
 We apply now Theorem \ref{thm:main} to this and consider the same wedge summand to get
 \begin{multline}\label{eqn:induceddiag}
 \widehat{Z}\big(K_{I};(\underline{C},\underline{E})_I\big)
\wedge\widehat{Z}\big(K_{P\cup Q- I};(\underline{B},\underline{B})_{P\cup Q-I}\big)\\
\xrightarrow{\phantom{m}\widehat{\chi}^{J,L}_{I}\phantom{m}}
\widehat{Z}\big(K_{I};(\underline{\widetilde{C}},\underline{\widetilde{E}})_{I}\big)
\wedge \widehat{Z}\big(K_{P\cup Q- I};(\underline{\widetilde{B}},\underline{\widetilde{B}})_{P\cup Q-I}\big).
\end{multline}
where  $\widehat{\chi}^{J,L}_{I}$ is induced from $\widehat{\psi}^{P,Q}_{P\cup Q}$ in \eqref{eqn:psiuv}   
via Theorem \ref{thm:main}. In order to analyze this further, we need a lemma.
\begin{lem}\label{lem:nomixing}
The diagram below commutes.
\begin{equation}\label{eqn:starequalscup2}
\hspace{0.0cm}\begin{tikzcd}[column sep = 3cm]
\widehat{Z}\big(K_{I};(\underline{C},\underline{E})_I\big)
\arrow[r, "\widehat{\chi}^{J,L}_{I}\big|_{ \widehat{Z}\big(K_{I};(\underline{C},\underline{E})_I\big)}"]
&\widehat{Z}\big(K_{I};(\underline{\widetilde{C}},\underline{\widetilde{E}})_{I}\big)\\ 
\widehat{Z}\big(K_{I};(\underline{C},\underline{E})_I\big)
\arrow[r, "\widehat{\psi}^{J,L}_{I}"] \arrow[u, "="'] 
&\widehat{Z}\big(K_{I};(\underline{C},\underline{E})_{J\cup L}^{J,L}\big)
\arrow[u, "\widehat{\phi}^{J,L}_{I}"]
\end{tikzcd}
\end{equation}
\nd where the map $\widehat{\phi}^{J,L}_{I}$ is induced by the map of pairs
{\small\begin{equation*}
\begin{cases} (C_i,E_i) \mapsto (C_i,E_i)
\phantom{m}\text{if}\;\; i \in  J\cup L \msetm J \cap L, 
\;\text{else if}\;\; i \in J\cap L,\\
\big(C_i  \wedge C_i, E_i\wedge E_i) \hookrightarrow  \big((C_i  \wedge C_i) \vee (B_i\wedge C_i) 
\vee (C_i \wedge B_i),\; \;(E_i \wedge E_i) \vee  (B_i \wedge E_i)\vee (B_i\wedge E_i)\big) \end{cases}
\end{equation*}}
the latter by the inclusion into the first wedge summands.
\end{lem}
\begin{proof}
This follows from the naturality of Theorem \ref{thm:main} for maps of wedge decomposable pairs.
\end{proof}
This lemma allow us now to replace \eqref{eqn:induceddiag} with the  map
\begin{multline}\label{eqn:okfordiags}
\widehat{Z}\big(K_{I};(\underline{C},\underline{E})_I\big)
\wedge\widehat{Z}\big(K_{P\cup Q- I};(\underline{B},\underline{B})_{P\cup Q-I}\big)\\
\xrightarrow{\phantom{n}\widehat{\psi}^{J,L}_{I}\wedge\; 
\widehat{\psi}^{P-J,Q-L}_{P\cup Q- I}\phantom{i}}
\;\widehat{Z}\big(K_{I};(\underline{C},\underline{E})_{J\cup L}^{J,L}\big)
\wedge \widehat{Z}\big(K_{P\cup Q- I};
(\underline{B},\underline{B})_{P\cup Q-I}^{P-J, Q-L}\big).
\end{multline}
\nd where we have used
\begin{equation}\label{eqn:bsthesame}
(\underline{B},\underline{B})_{P\cup Q-I}^{P-J, Q-L} =
(\underline{\widetilde{B}},\underline{\widetilde{B}})_{P\cup Q-I}.
\end{equation}

The next theorem describes now the product for the cohomology of a wedge decomposable pair.
\begin{thm}\label{thm:product for wdpair}
The partial diagonal map for a wedge decomposable pair $(U,V)$
$$\widehat{\Delta}_{P \cup Q}^{P,Q}\;\colon \widehat{Z}\big(K_{P\cup Q};(\underline{U},\underline{V})_{P\cup Q} \big) 
\longrightarrow \widehat{Z}\big(K_{P};(\underline{U},\underline{V})_{P} \big) \wedge 
\widehat{Z}\big(K_{Q};(\underline{U},\underline{V})_{ Q} \big),$$
is realized on each wedge summand given by Theorem \ref{thm:main} as
{\fontsize{10.5}{11}\selectfont \begin{multline}\label{eqn:smashofdiags}
\widehat{Z}\big(K_{I};(\underline{C},\underline{E})_I\big)
\wedge\widehat{Z}\big(K_{P\cup Q- I};(\underline{B},\underline{B})_{P\cup Q-I}\big)
\hspace{0.1truein}\xrightarrow{\phantom{n}\widehat{\Delta}^{J,L}_{I}\wedge\; \widehat{\Delta}^{P-J,Q-L}_{P\cup Q- I}\phantom{i}}\\
\big(\widehat{Z}(K_{J};(\underline{C},\underline{E})_{J})
\wedge \widehat{Z}(K_{P- J}(\underline{B},\underline{B})_{P-J})\big)
\wedge \big(\widehat{Z}(K_{L};(\underline{C},\underline{E})_{L})
\wedge \widehat{Z}(K_{Q- L}(\underline{B},\underline{B})_{Q- L})\big).
\end{multline}}
\end{thm}
\begin{proof}
The map is obtained in this form by using \eqref{eqn:bsthesame} to compose \eqref{eqn:okfordiags} with the shuffle map \eqref{eqn:shuffleforwd}.
\end{proof}
\skp{0.1}
This theorem enables  now direct calculation of the $\ast$-product on a wedge decomposable pair
by combining the calculation for $(\underline{C},\underline{E})$ given by Lemma \ref{lem:ce} with  calculation
for$(\underline{B},\underline{B})$ which is described next.

No links appear in the $\widehat{Z}\big(K_{P\cup Q- I};(\underline{B},\underline{B})_{P\cup Q- I}\big)$ wedge summand
and we have
have 
$$\widetilde{H}^{\ast}\big(\widehat{Z}(K_{P\cup Q- I};(\underline{B},\underline{B})_{P\cup Q- I})\big) \cong 
\motimes_{j \in {P\cup Q- I}}\widetilde{H}^{\ast}(B_j).$$
On this wedge factor, the partial diagonals are
\begin{multline}\label{eqn:bpdiag}
\widehat{\Delta}_{(P\cup Q-I)}^{P-J, Q-L} \;\colon 
 \widehat{Z}\big(K_{P\cup Q -I} ;(\underline{B},\underline{B})_{P\cup Q -I}\big) \\
 \longrightarrow
 \widehat{Z}\big(K_{P-J} ;(\underline{B},\underline{B})_{P-J}\big) \wedge 
 \widehat{Z}\big(K_{Q-L} ;(\underline{B},\underline{B})_{Q-L}\big),
\end{multline}
\nd inducing
\begin{equation}\label{eqn:bprod}
\motimes_{j \in P-J}\hspace{-2mm}\widetilde{H}^{\ast}(B_j) \otimes 
\motimes_{k\in Q-L}\hspace{-2mm}\widetilde{H}^{\ast}(B_l) 
\;\longrightarrow\hspace{-3mm} \motimes_{l \in P\cup Q-I}\hspace{-4mm} \widetilde{H}^{\ast}(B_l)
\end{equation}
\nd As in the case $(\underline{C},\underline{E})$, classes in the same cohomology factors are multiplied; these are the 
$\widetilde{H}^{\ast}(B_i)$ with $i \in  (P\msetm J) \cap (Q\msetm L)$, cf. \eqref{eqn:pqandi}. This is all summarized in the next 
corollary.
\begin{cor}\label{cor:wdecompprod}
The $\ast$-product associated to the ring $\widetilde{H}^{\ast}\big(Z(K;(\underline{U},\underline{V}))\big)$ for a wedge decomposable pair
$(\underline{U},\underline{V})$,  is determined by the smash product of partial diagonal maps \eqref{eqn:smashofdiags}. 
In cohomology the products are described by Lemma \ref{lem:ce} and by \eqref{eqn:bprod}. This determines  the cup product structure in 
$H^{\ast}\big(K;(\underline{U}, \underline{V})\big)$ by the description given in subsection \ref{subsec:back}.
\end{cor}

\subsection{The product structure for general CW pairs}\label{subsec:prodxa}
We wish now to extend these results about products, from wedge decomposable pairs
$(\underline{U},\underline{V}) = (\underline{B\vee C},\underline{B\vee E})$, to the cohomology of
$Z(K;(\underline{X}, \underline{A})$ for general pairs $(\underline{X},\underline{A})$.
Our main tool is Theorem \ref{thm:cartan2} which asserts that, given $(\underline{X}, \underline{A})$, we can find
wedge decomposable pairs $(\underline{U},\underline{V}) = (\underline{B\vee C},\underline{B\vee E})$ so that as 
groups, there is an isomorphism
\begin{equation}\label{eqn:generalequalsuv}
\theta_{(\underline{U},\underline{V})} \colon \widetilde{H}^\ast\big(Z(K;(\underline{U}, \underline{V})\big) \longrightarrow {H}^\ast\big(Z(K;(\underline{X}, \underline{A})\big).
\end{equation}
\nd Here \eqref{eqn:generalequalsuv} is used to label the generators only, the product structure will be in terms
of the product structure in the rings $\widetilde{H}^{\ast}(X_i)$ and $\widetilde{H}^{\ast}(A_i)$ and {\em not\/} in $\widetilde{H}^{\ast}(B_i)$,
$\widetilde{H}^{\ast}(C_i)$ or $\widetilde{H}^{\ast}(E_i)$. 

Unlike the previous example of wedge decomposable pairs, Lemma \ref {lem:nomixing} will not hold as 
the pairs $(\underline{X}, \underline{A})$ are not wedge decomposable in general.  In this new situation, products in 
$\widetilde{H}^{\ast}(X_i)$ and $\widetilde{H}^{\ast}(A_i)$ will mix the modules $B_i'$, $C_i'$ and $E_i'$
and so upset the links which appear in the partial diagonal maps \eqref{eqn:cediag} and \eqref{eqn:bpdiag}.
Our task then is to keep track of these changes.

We begin by examining the diagonal maps for the spaces $X_i$ and $A_i$ and look at the diagonal maps
\eqref{eqn:diag} in terms of the notation adopted in Definition \ref{defn:strongfc}. For the long exact
sequence
\begin{equation*}\label{eqn:lesdiag}
\overset{\delta}{\to} \widetilde{H}^\ast(X_i\wedge X_i\big/A_i\wedge A_i) \overset{\ell}{\to} 
\widetilde{H}^\ast(X_i\wedge X_i) \overset{\iota}{\to} \widetilde{H}^*(A_i\wedge A_i)  
\overset{\delta}{\to} \widetilde{H}^{\ast+1}(X_i\wedge X_i\big/A_i\wedge A_i) \to  
\end{equation*}
\nd the appropriate image, kernel and cokernel modules are given as follows. 
\nd \skp{-0.15}
\begin{enumerate}\itemsep1.3mm
\item $\widetilde{H}^{\ast}(A_i\wedge A_i)\cong \check{B}_i\oplus \check{E}_i$
\item  $\widetilde{H}^{\ast}(X_i\wedge X_i)\cong \check{B}_i \oplus \check{C}_i$,  \; where 
$\check{B}_i \underset{\simeq}{\overset{\iota}{\to}} \check{B}_i, \;\;\left.\iota\right|_{\check{C}_i} = 0$ 
\item  $\widetilde{H}^{\ast}(X_i\wedge X_i\big/A_i\wedge A_i)\cong \check{C}_i\oplus W_i',$ \;
where $\check{C}_i \underset{\simeq}{\overset{\ell}{\to}}\check{C}_i, \;\; \left.\ell\right|_{\check{B}_i} = 0, \;\; 
\check{E}_i \underset{\simeq}{\overset{\delta}{\to}} W_i'$. 
\end{enumerate}

Comparing these to the graded $k$-modules we have for the pair $(\underline{X},\underline{A})$ we get
\begin{align}\label{eqn:newbcexa}
\check{B}_i&= B_i' \otimes B_i'\nonumber\\
\check{C}_i&= C_i' \otimes C_i'\; \oplus\;  C_i' \otimes B_i' \;\oplus\;
 B_i' \otimes C_i'\\
\check{E}_i&= E_i' \otimes E_i' \;\oplus\;  E_i' \otimes B_i' \;\oplus\;
 B_i' \otimes E_i'.\nonumber
\end{align}

\nd {\bf Note:} These cohomology modules are realized now by the cohomology of the spaces 
$\widehat{B}_i$, $\widehat{C}_i$ and $\widehat{E}_i$ as defined in \eqref{eqn:newbce}, and we continue to adopt 
the notation from Definition \ref{defn:tildepairs}.  Also, in order to keep the notation as simple as possible, we shall
supress explicit mention of the additive isomorphism $\theta_{(\underline{U},\underline{V})}$ from 
\eqref{eqn:generalequalsuv}, though its usage will be understood throughout the remainder of this section. 

Consider again the part of the partial diagonals given by the shuffle 
maps from \eqref{eqn:composite2}
\begin{equation}\label{eqn:shufflemap2}
 \widehat{Z}\big(K_{P\cup Q};(\underline{X},\underline{A})^{P,Q}_{P\cup Q}\big)    
 \xrightarrow{\widehat{S}}  \widehat{Z}\big(K_{P};(\underline{X},\underline{A})_{P} \big) \wedge 
\widehat{Z}\big(K_{Q};(\underline{X},\underline{A})_{ Q} \big).
\end{equation}

\nd Applying Theorem \ref{thm:cartan2}, the shuffle map for $(\underline{X},\underline{A})$ can
be described {\em additively\/} in cohomology, in terms of wedge decomposable pairs, as follows
\begin{equation}\label{eqn:therealgroups}
\widetilde{H}^{\ast}\big(\widehat{Z}\big(K_{P};(\underline{U},\underline{V})_{P} \big) \wedge 
\widehat{Z}\big(K_{Q};(\underline{U},\underline{V})_{Q}\big)\big)
\xrightarrow{\widehat{S}}
\widetilde{H}^{\ast}\big(\widehat{Z}\big(K_{P\cup Q};(\underline{U},\underline{V})^{P,Q}_{P\cup Q}\big)\big),
\end{equation}
where as in \eqref{eqn:uvshufflewd} we have, (in the notation of Definition \ref{defn:tildepairs}) 
\begin{equation}\label{eqn:uvwd}
\widehat{Z}\big(K_{P\cup Q};(\underline{U},\underline{V})^{P,Q}_{P\cup Q}\big)  \;=\;
\widehat{Z}\big(K_{P\cup Q}; (\widetilde{B}\vee \widetilde{C},\widetilde{B}\vee \widetilde{E})_{P\cup Q} \big).
\end{equation}
Figure\;1 below is a useful aid in visualizing all that follows. It displays all sets, simplices and modules which 
are relevant to the discussion of the shuffle map for CW pairs and their cohomlogical wedge decompositions.
(The numbers in square brackets label regions of the Venn diagram non-uniquely.)
\skp{0.4}
\nd \hspace{-0.1in}\includegraphics[width=6.5in]{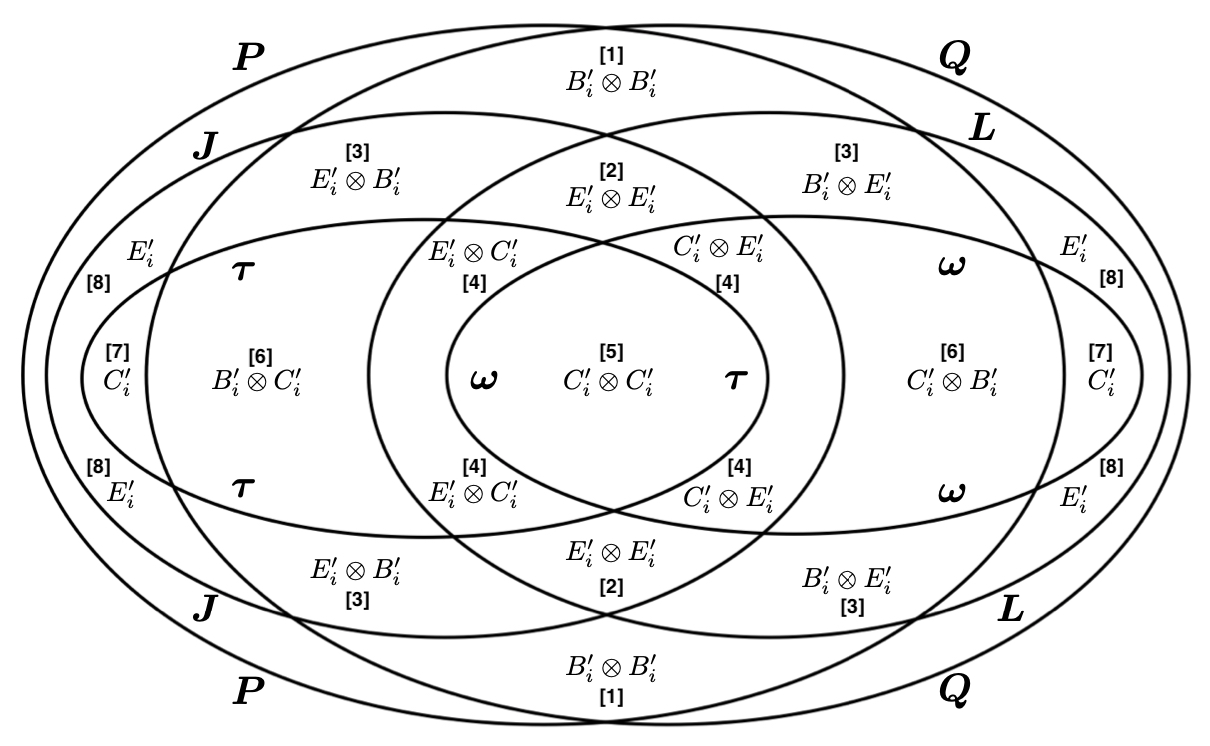}
\skp{0.3}
\centerline{{\sc Figure 1}}
\skp{0.1}
\centerline{The arrangement of modules in}
\centerline{{\fontsize{11}{11}\selectfont 
$H^{\ast}\big(\widehat{Z}\big(K_{P\cup Q};(\underline{U},\underline{V})^{P,Q}_{P\cup Q}\big)\big) =
H^{\ast}\big(\widehat{Z}\big(K_{P\cup Q}; (\widetilde{B}\vee \widetilde{C},\widetilde{B}\vee \widetilde{E})_{P\cup Q} \big)\big)$}}
\skp{0.3}
Next, we apply Theorem \ref{thm:main} to the spaces
$\widehat{Z}\big(K_{P};(\underline{U},\underline{V})_{P}$ and
$\widehat{Z}\big(K_{Q};(\underline{U},\underline{V})_{Q}\big)$
which appear on the left hand side of \eqref{eqn:therealgroups}, and then take the cohomology of each of the wedge 
summands resulting, to get the cohomological description of \eqref{eqn:shufflemap2} below,
\begin{multline}\label{eqn:cohomshuffle2}
\hspace{0.4truein}\widetilde{H}^{\ast}\big(\widehat{Z}\big(K_{J};(\underline{C},\underline{E})_J\big)\big) \otimes
\widetilde{H}^{\ast}\big(\widehat{Z}\big(K_{P- J};(\underline{B},\underline{B})_{P-J}\big)\big)\\
\otimes \widetilde{H}^{\ast}\big(\widehat{Z}\big(K_{L};(\underline{C},\underline{E})_L \big)\big)
\otimes  \widetilde{H}^{\ast}\big(\widehat{Z}\big(K_{Q- L};(\underline{B},\underline{B})_{Q-L}\big)\big)\\
\xrightarrow{\phantom{m}\widehat{S}^{\ast}\phantom{m}} \;
\widetilde{H}^{\ast}\big(\widehat{Z}\big(K_{P\cup Q}; (\widetilde{B}\vee \widetilde{C},\widetilde{B}\vee 
\widetilde{E})_{P\cup Q} \big) \big).
\end{multline}
The next lemma is now apposite. As usual, (Definition \ref{defn:tildepairs}), we set $I = J\cup L$.
\begin{lem}\label{lem:wedgeslineup}
In \eqref{eqn:cohomshuffle2} the target of the shuffle map $\widehat{S}^\ast$ is the summand
\begin{equation}\label{eqn:shufflesummand}
\widetilde{H}^{\ast}\big(\widehat{Z}\big(K_{I};(\widetilde{\underline{C}},\widetilde{\underline{E}})_I \big)\big)
\otimes  \widetilde{H}^{\ast}\big(\widehat{Z}\big(K_{P\cup Q- I};(\widetilde{\underline{B}},\widetilde{\underline{B}})_{P\cup Q- I}\big)\big)
\end{equation}
in the decomposition of $\widetilde{H}^{\ast}\big(\widehat{Z}\big(K_{P\cup Q}; (\widetilde{B}\vee \widetilde{C},\widetilde{B}\vee 
\widetilde{E})_{P\cup Q} \big) \big)$ given by Theorem \ref{thm:main}.
\end{lem}
\begin{rem} This is not obvious from \eqref{eqn:shuffleforwd} because it is at the cohomology level only that
we can replace the shuffle map in which we are interested, \eqref{eqn:shufflemap2}, with one involving wedge 
decomposable pairs \eqref{eqn:therealgroups}.\end{rem}
\begin{proof}
Consider again the equality of sets \eqref{eqn:pqandi}, (cf. Figure\;1),
$$(P\cup Q) \msetm I \;=\; \big((P\msetm J)\cup (Q\msetm L)\big) \msetm
\big(J\cap (Q\msetm L)\big) \cup \big(L\cap (P\msetm J)\big).$$
All the factors $H^{\ast}(B_i)$ on the left hand side of \eqref{eqn:cohomshuffle2} satisfy $i \in (P\msetm J)\cup (Q\msetm L)$.
The only ones which can be paired with $H^{\ast}(C_i)$  or $H^{\ast}(E_i)$ are those for which $i$ is in the disjoint union
$\big(J\cap (Q\msetm L)\big) \cup \big(L\cap (P\msetm J)\big) \subset I$. In this case,
$$H^{\ast}(C_i)\otimes H^{\ast}(B_i) \subset H^{\ast}(\widehat{C}_i)\;\;\text{and}\;\; 
H^{\ast}(E_i)\otimes H^{\ast}(B_i) \subset H^{\ast}(\widehat{E}_i)$$
So, the factors of $H^{\ast}(B_i)$ from the left hand side of \eqref{eqn:cohomshuffle2} which are ``lost'' are precisely the 
ones for which $i \in I$. All other copies of $H^{\ast}(B_i)$ survive to appear in the right hand tensor factor of
\eqref{eqn:shufflesummand}.
\end{proof}
The main theorem of this section is next.
\begin{thm}\label{thm:mainprodthm}
 The additive isomorphism $\theta_{(\underline{U},\underline{V})}$ of \eqref{eqn:generalequalsuv} 
suffices to determine explicitly the partial diagonal map
\eqref{eqn:partialdiags}
\begin{equation*}
\widetilde{H}^{\ast}\big(\widehat{Z}\big(K_{P};(\underline{X},\underline{A})_{P} \big) \wedge 
\widehat{Z}\big(K_{Q};(\underline{X},\underline{A})_{ Q} \big)\big)
\xrightarrow{(\widehat{\Delta}_{P \cup Q}^{P,Q})^{\ast}} 
\widetilde{H}^{\ast}\big(\widehat{Z}\big(K_{P\cup Q};(\underline{X},\underline{A})_{P\cup Q} \big)\big),
\end{equation*}
and hence the cup product structure in $H^{\ast}\big(K;(\underline{X}, \underline{A})\big)$, by the description given
in subsection \ref{subsec:back}.
\end{thm}
\nd The rest of this section is devoted to the proof of Theorem \ref{thm:mainprodthm}.
Consider a class $x_P\otimes x_Q$ in the left hand side of \eqref{eqn:therealgroups}
\begin{equation}
x_P\otimes x_Q \in \widetilde{H}^{\ast}\big(\widehat{Z}\big(K_{P};(\underline{U},\underline{V})_{P} \big) \wedge 
\widehat{Z}\big(K_{Q};(\underline{U},\underline{V})_{Q}\big)\big).
\end{equation}
Theorem \ref{thm:main} describes this class as a sum of terms of the form $u\otimes v$ in 
{\small\begin{multline}\label{eqn:summandofxpxq}
\widetilde{H}^{\ast}\big(\widehat{Z}\big(K_{J};(\underline{C},\underline{E})_J\big)\big) \otimes
\widetilde{H}^{\ast}\big(\widehat{Z}\big(K_{P- J};(\underline{B},\underline{B})_{P-J}\big)\big)\\
\otimes \widetilde{H}^{\ast}\big(\widehat{Z}\big(K_{L};(\underline{C},\underline{E})_L \big)\big)
\otimes  \widetilde{H}^{\ast}\big(\widehat{Z}\big(K_{Q- L};(\underline{B},\underline{B})_{Q-L}\big)\big),
\end{multline}}
each of which has the form below by Corollary \ref{cor:wedge}
\begin{enumerate}\itemsep1.5mm
\item[] $\ds{u = \alpha \otimes  \motimes_{i\in \tau}c_i \otimes \motimes_{j\in J\msetm \tau}e_j \otimes\motimes_{j \in P\msetm J}b_j
\;\;\in\widetilde{H}^{\ast}(\Sigma{|lk_{\tau}K_J|}\; \otimes\; \widetilde{H}^{\ast}(\widehat{D}_{\underline{C},\underline{E}}^{J}(\tau))
\motimes_{j \in P\msetm J}\hspace{0mm}\widetilde{H}^{\ast}(B_j)}$ 
\item[] $\ds{v = \beta \otimes  \motimes_{i\in \omega}c_i \otimes 
\motimes_{j\in L\msetm \omega}e_j  \otimes \motimes_{j \in Q\msetm L}b_j  \;\;\in \widetilde{H}^{\ast}(\Sigma{|lk_{\omega}K_L|}\;
\otimes\; \widetilde{H}^{\ast}(\widehat{D}_{\underline{C},\underline{E}}^{L}(\omega)) 
\motimes_{j \in Q\msetm L}\hspace{0mm}\widetilde{H}^{\ast}(B_j)}$ 
\end{enumerate}
\skp{0.2}
It will be convenient to call $\alpha$ and $\beta$ {\em indexing links\/}. Next, we apply the shuffle map in cohomology \eqref{eqn:cohomshuffle2}, ignoring cohomological degree and keeping in mind 
\eqref{eqn:newbce} and Figure\;1. This gives the description below in which cohomology elements are labelled by the regions
in Figure\;1 to which they belong.  Recall that \eqref{eqn:diag} restricts cohomology products to the set $ P\cap Q$ only, in particular,
to Regions [1], [2], [3], [5] and [5] of Figure\;1 only. From  \eqref{eqn:newbcexa} we see that products on Region [4] are not 
supported. (Note that some relabelling in the names of classes is necessary to avoid ambiguities.)
\begin{align}\label{eqn:tensorofclasses}
\nonumber\hspace{0.2truein}\widehat{S}^{\ast}(u\otimes v) =&\;\widehat{S}^{\ast}(\alpha\otimes \beta)\\ 
&\otimes\motimes_{i \in [5]}(c_i\otimes \bar{c}_i)
\otimes\motimes_{j \in [6]}(c_j\otimes b_j)
\otimes\motimes_{k \in [2]}(e_k\otimes \bar{e}_k)
\otimes\motimes_{l \in [3]}(e_l\otimes b_l) \\ \nonumber
&\otimes\motimes_{s\in [1]}(b_s\otimes \bar{b}_s)
\hspace{0.15in}\otimes\motimes_{\{t,u,v\}\;{\rm otherwise}}(c_t\otimes e_u \otimes b_v) 
\end{align}
\nd in the group
\begin{equation}\label{eqn:targetofshuffle}
\widetilde{H}^{\ast}\big(\widehat{Z}\big(K_{P\cup Q};(\underline{U},\underline{V})^{P,Q}_{P\cup Q}\big)\big)  \;=\;
\widetilde{H}^{\ast}\big(\widehat{Z}\big(K_{P\cup Q}; (\widetilde{B}\vee \widetilde{C},\widetilde{B}\vee 
\widetilde{E})_{P\cup Q} \big)\big).
\end{equation}
Notice that by \eqref{eqn:newbcexa} no products arising from Region [4] in Figure\;1 can be supported.
Next, we compose with the map induced by \eqref{eqn:composite1},
\begin{equation}\label{eqn:psi}
\widetilde{H}^{\ast}\big(\widehat{Z}\big(K_{P\cup Q};(\underline{X},\underline{A})^{P,Q}_{P\cup Q}\big)\big) 
\xrightarrow{(\widehat{\psi}^{P,Q}_{P\cup Q})^{\ast}} 
\widetilde{H}^{\ast}\big(\widehat{Z}\big(K_{P\cup Q};(\underline{X},\underline{A})_{P\cup Q} \big)\big) 
\end{equation}
to get the full partial diagonal map induced by \eqref{eqn:partialdiags},

\begin{equation}
\widetilde{H}^{\ast}\big(\widehat{Z}\big(K_{P};(\underline{X},\underline{A})_{P} \big) \wedge 
\widehat{Z}\big(K_{Q};(\underline{X},\underline{A})_{ Q} \big)\big)
\xrightarrow{(\widehat{\Delta}_{P \cup Q}^{P,Q})^{\ast}} 
\widetilde{H}^{\ast}\big(\widehat{Z}\big(K_{P\cup Q};(\underline{X},\underline{A})_{P\cup Q} \big)\big) 
\end{equation}
and hence a computation of 
\begin{equation}
u\ast v = 
(\widehat{\Delta}_{P \cup Q}^{P,Q})^{\ast}(u\otimes v) =
\big((\widehat{\psi}^{P,Q}_{P\cup Q})^{\ast}\circ
\widehat{S}^{\ast}\big)(u\otimes v)
\end{equation}
which is the part of $x_P \ast x_Q$ in the summand \eqref{eqn:summandofxpxq}. 
The homomorphism $\widehat{\psi}^{P,Q}_{P\cup Q})^{\ast}$ multiplies the terms in
$\widehat{S}^{\ast}(u\otimes v)$ corresponding to the marked regions.
The consequences are discussed below. Recall that the target of the map 
$(\widehat{\psi}^{P,Q}_{P\cup Q})^{\ast}$, \eqref{eqn:psi}, is given by Corollary \ref{cor:wedge},
{\fontsize{11}{11}\selectfont \begin{multline*}
\widetilde{H}^{\ast}\big(\widehat{Z}(K(_{P\cup Q};(\underline{X},\underline{A}\big)_{P\cup Q}))\;
\xrightarrow{\cong}\; \widetilde{H}^{\ast}\big(\widehat{Z}(K_{P\cup Q};(\underline{U},\underline{V})_{P\cup Q})\big)\\
\xrightarrow{\cong}
\moplus_{I\subset P\cup Q}\big(\moplus_{\sigma \in K_{I}}\widetilde{H}^{\ast}(\Sigma{|lk_{\sigma}(K_{I})|})
\otimes \widetilde{H}^{\ast}\big(\widehat{D}_{\underline{C},\underline{E}}^{I}(\sigma))\big)
\motimes  \widetilde{H}^{\ast}\big(\widehat{Z}(K_{(P\cup Q)\msetm I};(\underline{B},\underline{B})_{[m]-I})\big).
\end{multline*}}

Next we analyze the monomials which can appear in 
$(\widehat{\Delta}_{P \cup Q}^{P,Q})^{\ast}(u\otimes v)$. 
The properties of the modules $B_i', C_i'$ and $E_i'$ from Definition \ref{defn:strongfc}
are used strongly below. All cup products are either in $\widetilde{H}^{\ast}(X_i) \cong B_i'\oplus C_i'$ or in 
$\widetilde{H}^{\ast}(A_i) \cong B_i'\oplus E_i'$. 
\skp{0.2}
Below is a list of all cup products which can occur when $(\widehat{\psi}^{P,Q}_{P\cup Q})^{\ast}$  is applied to
the class $\widehat{S}^{\ast}(u\otimes v)$ in the group \eqref{eqn:targetofshuffle}:
\skp{0.3}
\begin{enumerate}\itemsep2mm
\item For $i \in$ Region $[5]$, $c_i\otimes \bar{c}_j \mapsto c_{i}^{[5]} \in C_i'$.
\item For $j \in$ Region $[6]$, $c_i\otimes b_j \mapsto c_{j}^{[6]} \in C_j'$.
\item For $k \in$ Region $[2]$, $e_k\otimes \bar{e}_k \mapsto e_{k}^{[2]} + b_k^{[2]} \in E_k'\oplus B_k'$.
\item\label{itm:fourth} For $l \in$ Region $[3]$, $e_l\otimes b_l \mapsto e_{l}^{[3]} + b_l^{[3]} \in E_l'\oplus B_l'$.
\item For $s\in$ Region $[1]$, $b_s\otimes \bar{b}_s \mapsto b_{s}^{[1]} + c_s^{[1]} \in B_s'\oplus C_s'$.
\end{enumerate}
\skp{0.2}
\nd The next two lemmas will allow us to keep track of the links indexing the monomials.

\begin{lem}
Let $K$ be a simplicial complex on vertices $[m]$, $I \subset [m]\msetm \{s\}$, $\sigma \in K_I$ and
$\sigma \cup \{s\} \in K_{I\cup \{s\}}$. Then there is a natural map
\begin{equation}
\begin{array}{lccc}
\rho_{I,s} \colon\hspace{-3mm}&lk_{\sigma \cup \{s\}}K_{I\cup \{s\}} & \longrightarrow & lk_{\sigma}K_I\\[0.3mm]
                        &\tau           & \mapsto     &\tau \cap I
\end{array}
\end{equation}
\end{lem}

\nd{\em Proof:\/}  Let $\tau \in lk_{\sigma \cup \{s\}}K_{I\cup \{s\}}$, then $\tau \cup (\sigma \cup \{s\}) \in K_{I\cup \{s\}}$
and so $\tau \cup \sigma \in K_{I\cup \{s\}}$ implying $(\tau \cap I) \cup \sigma \in K_I$. \hspace{4.8truein}
    $\square$

\begin{lem}\label{lem:linkchange2}
Let $K$ be a simplicial complex on vertices $[m]$, $I \subset [m]$, $\sigma \in K_I$ and
and $l \in I$. Then there is a natural inclusion,
\begin{equation}
\begin{array}{lccc}
\iota_l \colon\hspace{-3mm}&lk_{\sigma}K_{I\msetm \{l\}} & \longrightarrow & lk_{\sigma}K_I.                        
\end{array}
\end{equation}
and the diagram below commutes.
\begin{equation}\label{eqn:linkinclusion1}
\hspace{0.0cm}\begin{tikzcd}
lk_{\sigma}K_{I\msetm \{l\}}
\arrow[r, "\iota_l"]
&lk_{\sigma}K_I\\ 
lk_{\sigma\cup \{s\}}K_{(I\cup \{s\})\msetm \{l\}}
\arrow[r, "\iota_l"] \arrow[u, "\rho_{I\msetm \{l\},s}"'] 
&lk_{\sigma \cup \{s\}}K_{I\cup \{s\}}
\arrow[u, "\rho_{I, \{s\}}"]
\end{tikzcd}
\end{equation}
 \end{lem} 

\nd{\em Proof:\/} Notice first that
$$lk_{\sigma\cup \{s\}}K_{(I\cup \{s\})\msetm \{l\}} = lk_{\sigma\cup \{s\}}K_{(I\msetm \{l\})\cup \{s\}}.$$

\nd because $I\cup \{s\})\msetm \{l\} = (I\msetm \{l\})\cup \{s\}$.  
Now let $\tau \in k_{\sigma\cup \{s\}}K_{(I\cup \{s\})\msetm \{l\}}$, then 
$$(\rho_{I,s} \circ \iota_{l})(\tau) = \rho_{I,s}(\tau) = \tau\cap I.$$
\nd On the other hand, 
$$(\iota_{l} \circ \rho_{I\msetm l,s})(\tau) = \iota_{l}(\tau\cap I) = \tau\cap I.$$
\nd as well.\hspace{5.6truein}$\square$

We are in a position now to enumerate and describe the monomials, including their indexing links, which arise from 
\eqref{eqn:tensorofclasses}. Recall that
$\big((\widehat{\psi}^{P,Q}_{P\cup Q})^{\ast}\circ
\widehat{S}^{\ast}\big)(\alpha \otimes \beta) = \alpha\ast \beta$ is described in subsection \ref{subsec:linkprods}. 
 The subscripts below are preserved from \eqref{eqn:tensorofclasses}. 

\nd \begin{enumerate}[(a)]
\item $(\alpha\ast \beta)\otimes c_{i}^{[5]} \otimes c_{j}^{[6]} \otimes e_{k}^{[2]} \otimes e_{l}^{[3]} \otimes b_{s}^{[1]}$
\item[] There are no changes to the indexing link $\alpha\ast \beta$ here.
\item[]
\item $\rho_{I,s}^{\ast}( \alpha \ast \beta) \otimes c_{i}^{[5]} \otimes c_{j}^{[6]} \otimes 
e_{k}^{[2]} \otimes e_{l}^{[3]} \otimes c_{s}^{[1]}$
\item[]  The class  $c_s^{[1]}$ in item (5) above will be zero unless the simplex $\nu = \sigma\cup \{s\}$ exists 
in $K_{I\cup \{s\}}$. In which case, the indexing link $\alpha\ast \beta$ must change to $\rho_{I,s}^{\ast}(\alpha\ast \beta)$.
Here, the simplex $\nu = \sigma \cup \{s\}$ exists in $K_{I\cup \{s\}}$ and so $\nu$ is a full subcomplex of $K$, in which case,
$Z\big(\nu;(X,A)\big)= \mprod_{i_k\in \nu}X_{i_k}$ 
\nd retracts off $Z(K;(X,A))$ and the product is occurring in the subring $\ds{H^{\ast}\big(\mprod_{i_k\in \nu}X_{i_k}\big)}$.
\item[]
\item\label{itm:third}$\iota_l^{\ast}(\alpha\ast \beta)\otimes c_{i}^{[5]} \otimes c_{j}^{[6]} \otimes e_{k}^{[2]} 
\otimes b_{l}^{[3]} \otimes b_{s}^{[1]}$.
\item[]  In this case the appearance of the class $b_{l}^{[3]}$ from item (\ref{itm:fourth}) above removes the factor 
$\widetilde{H}^{\ast}(E_l)$ from $\widetilde{H}^{\ast}\big(\widehat{D}_{\underline{C},\underline{E}}^{I}(\sigma)\big)$,
which was supporting the class $e_l^{[3]}$. This changes the set $I \subset P\cup Q$ to the set $I\msetm \{l\}$ and so
the indexing link $\alpha\ast \beta$ changes to $\iota^{\ast}_{l}(\alpha\ast \beta)$.
\item[]
\item $(\iota_{l} \circ \rho_{I\msetm l,s})^{\ast}(\alpha\ast \beta)\otimes c_{i}^{[5]} \otimes c_{j}^{[6]} 
\otimes e_{k}^{[2]} \otimes b_{l}^{[3]} \otimes c_{s}^{[1]}$
\item[]  The class  $c_s^{[1]}$ in item (5) above will be zero unless the simplex $\nu = \sigma\cup \{s\}$ exists 
in $K_{I\cup \{s\}}$. We have also the appearance of $b_{l}^{[3]}$ as in item (\ref{itm:third}), necessitating a change from
$I$ to $I\msetm\{l\}$. According to Lemma \ref{lem:linkchange2}, the indexing link changes from $\alpha\ast \beta$
to $(\iota_{l} \circ \rho_{I\msetm l,s})^{\ast}(\alpha\ast \beta)$.
\item[]
\item\label{itm:fifth} $\iota_k^{\ast}(\alpha\ast \beta)\otimes c_{i}^{[5]} \otimes c_{j}^{[6]} \otimes b_{k}^{[2]} \otimes e_{l}^{[3]} \otimes b_{s}^{[1]}$
\item[]  This time the link is altered by the appearance of $b_{k}^{[2]}$ replacing a class in $E_k'$, so $I$ is 
replaced with $I\msetm\{k\}$ and the link $\alpha\ast \beta$ is changed to $\iota_k^{\ast}(\alpha\ast \beta)$.
\item[]
\item $(\iota_{k} \circ \rho_{I\msetm k,s})^{\ast}(\alpha\ast \beta)\otimes c_{i}^{[5]} \otimes c_{j}^{[6]} \otimes b_{k}^{[2]} 
\otimes e_{l}^{[3]} \otimes c_{s}^{[1]}$
\item[]  The class  $c_s^{[1]}$ in item (5) above will be zero unless the simplex $\nu = \sigma\cup \{s\}$ exists 
in $K_{I\cup \{s\}}$. We have also the appearance of $b_{k}^{[2]}$ as in item (\ref{itm:fifth}), necessitating a change from
$I$ to $I\msetm\{k\}$. According to Lemma \ref{lem:linkchange2}, the indexing link changes from $\alpha\ast \beta$
to $(\iota_{k} \circ \rho_{I\msetm k,s})^{\ast}(\alpha\ast \beta)$.
\item[]
\item $(\iota_{k} \circ\iota_{l})^{\ast}(\alpha\ast \beta)\otimes c_{i}^{[5]} \otimes c_{j}^{[6]} 
\otimes b_{k}^{[2]} \otimes b_{l}^{[3]} \otimes b_{s}^{[1]}$
\item[]   This time we have the terms $b_{k}^{[2]}$ and $b_{l}^{[3]}$ appearing, requiring  a change from
$I$ to $I\msetm\{l,k\}$. The indexing link changes from $\alpha\ast \beta$
to $(\iota_{k} \circ\iota_{l})^{\ast}(\alpha\ast \beta)$.
\item[]
\item $c_{i}^{[5]} \otimes c_{j}^{[6]} \otimes b_{k}^{[2]} \otimes b_{l}^{[3]} \otimes c_{s}^{[1]}$
\item[]  Here, all three link changing terms  $b_{k}^{[2]}, b_{l}^{[3]}$ and $c_{s}^{[1]}$
appear. In the manner above, this alters the link from  $\alpha\ast \beta$ to 
$(\iota_l \circ \iota_k \circ \rho_{I\msetm \{l,k\},s})^{\ast}(\alpha\ast \beta)$.
\end{enumerate}
This completes the proof of Theorem \ref{thm:mainprodthm}.

\subsection{An example}\label{subsec:exm}
The methods of subsection \ref{subsec:prodxa} are used now to compute multiplicative structure in
Example \ref{exm:cp4additive} for the case $K = K_2$ the simplicial complex consisting of two
discrete points, so that $[m] = [2]$. This is the polyhedral product $Z(K_2;(M_f, \mathbb{C}P^3)\big)$, where 
$M_f$ is the mapping cylinder of the map
\begin{equation}
 f\colon \mathbb{C}P^3 \rightarrow \mathbb{C}P^{3}/\mathbb{C}P^1 
 \xhookrightarrow{\iota} \mathbb{C}P^{8}/\mathbb{C}P^1
 \end{equation}\label{eqn:mapcyl2}
 and the map $\iota$ is the inclusion of the bottom two cells.  Here
 $$\widetilde{H}^{\ast}(M_f) \cong k\{b^{4},b^{6},c^{8},c^{10},c^{12},c^{14},c^{16}\} \;\;{\rm and}\;\;
 \widetilde{H}^{\ast}(\mathbb{C}P^{3}) \cong k\{e^{2},b^{4},b^{6}\},$$
 \nd where {\em cohomological degree is denoted by a superscript\/}. According to (Definition \ref{defn:strongfc}), we have
 \begin{enumerate}\itemsep2mm
 \item[]  $B'_1 = k\{b_1^4,b_1^6\}, \;B'_2 = k\{b_2^4,b_2^6\}$
 \item[]  $C'_1 = k\{c_1^{8},c_1^{10},c_1^{12},c_1^{14},c_1^{16}\}, \; C'_2 = k\{c_2^{8},c_2^{10},c_2^{12},c_2^{14},c_2^{16}\}$
 \item[]  $E'_1 = k\{e_1^2\}, \;E'_2 = k\{e_2^2\}$
 \end{enumerate}
\nd where the subscript corresponds to the vertex and distinguishes the two different copies of the modules.
Note also that in $\widetilde{H}^{\ast}(M_f)$ there is the non-trivial cup product $c_i^8 \smile c_i^8 = c_i^{16}$ and in 
$\widetilde{H}^{\ast}(\mathbb{C}P^{3})$ we have $e_i^2\smile b_i^4 = b_i^6$.

We consider the case $P = Q = \{1,2\}$, $J = \{1\}$ and $L= \{1,2\}$, $\sigma = \tau = \omega = \{1\}$ and begin by 
determining the shuffle map $\widehat{S}^{\ast}$ on terms in \eqref{eqn:summandofxpxq}, $P\msetm J = \{2\}$ and
$Q\msetm L = \varnothing$,
\begin{align*}
&{ u = \alpha_{1,\{1\}}\otimes c_1^8 \otimes b_2^4} \;\in
\widetilde{H}^{\ast}(|\Sigma{lk_{\{1\}}(K_{\{1\}}|)}\otimes 
\widetilde{H}^{\ast}(\widehat{D}_{\underline{C},\underline{E}}^{\{1\}}(\{1\})) \otimes \widetilde{H}^{\ast}(B_2)\\
&{ v =\alpha_{1,\{1,2\}}\otimes c_1^8 \otimes e_2^2} \;\in
\widetilde{H}^{\ast}(|\Sigma{lk_{\{1\}}(K_{\{1,2\}}|)}\otimes 
\widetilde{H}^{\ast}(\widehat{D}_{\underline{C},\underline{E}}^{\{1,2\}}(\{1\}))
\end{align*}
\nd where here, the groups are summands of 
$\widetilde{H}^{\ast}\big(\widehat{Z}\big(K_{P};(\underline{U},\underline{V})_{P} \big)$ and 
$\widehat{Z}\big(K_{Q};(\underline{U},\underline{V})_{Q}\big)\big)$ given by Theorem \ref{thm:main}.
We apply now the shuffle map as in \eqref{eqn:tensorofclasses} to  get
\begin{align*}
{ \widehat{S}^{\ast}(u \otimes v)} &= 
\widehat{S}^{\ast}(\alpha_{1,\{1\}}\otimes \alpha_{1,\{1,2\}}) \otimes (c_1^8 \otimes c_1^8) \otimes
(b_2^4\otimes e_2^2)\\
&=\alpha_{1,\{1,2\}}\otimes (c_1^8 \otimes c_1^8) \otimes
(b_2^4\otimes e_2^2)
\end{align*}
\nd by \eqref{eqn:firstalpha}.
Next we apply the map $(\widehat{\psi}^{P,Q}_{P\cup Q})^{\ast}$ \eqref{eqn:psi}, which evaluates the cohomological 
diagonals, to give the case of item(\ref{itm:third}) above because $e_i^2\smile b_i^4 = b_i^6$. We get
$${ \widehat{\Delta}_{P \cup Q}^{P,Q})^{\ast}(u \otimes v)} = \iota^{\ast}_2(\alpha_{1,\{1,2\}})\otimes c_1^{16}\otimes b_2^{6}
= \alpha_{1,\{1\}}\otimes c_1^{16}\otimes b_2^{6}.$$

\bibliographystyle{amsalpha}

\end{document}